\documentclass[12pt, letterpaper]{amsart}
\usepackage[left=1in,right=1in,bottom=0.5in,top=0.8in]{geometry}
\usepackage{amsfonts}
\usepackage{amsmath, amssymb}
\usepackage{graphicx}
\usepackage[font=small,labelfont=bf]{caption}
\usepackage{epstopdf}
\usepackage{xcolor}
\usepackage{amsthm}
\usepackage{float}
\usepackage{pgfplots}
\usepackage{listings}
\usepackage{longtable}
\usepackage{mathrsfs}
\usepackage{bbold}
\usepackage{comment}
\usepackage{esint}
\usetikzlibrary{arrows, patterns, patterns.meta, calc, math}

\usepackage{mathtools}
\usepackage{scalerel}[2014/03/10]
\usepackage[usestackEOL]{stackengine}
\usepackage[utf8]{inputenc}
\usepackage[T1]{fontenc}
\usepackage[shortlabels, inline]{enumitem}
\DeclarePairedDelimiterX\set[1]\lbrace\rbrace{#1}
\usepackage{nicefrac}

\usepackage{scalerel}[2014/03/10]
\usepackage[usestackEOL]{stackengine}
\usepackage{enumitem}

\usepackage[backref=page]{hyperref}
\hypersetup{
plainpages=false,
colorlinks,
linkcolor={cyan!90!black},
citecolor={magenta},
urlcolor={red!90!black},
bookmarksdepth=2,}

\hypersetup{
pdftitle={Boundedness of the parabolic Riesz Transform},
pdfsubject={Mathematics, PDE, Analysis},
pdfauthor={Khalid Baadi},
pdfkeywords={}
}

\newtheorem{thm}{Theorem}[section]
\newtheorem{cor}[thm]{Corollary}
\newtheorem{prop}[thm]{Proposition}
\newtheorem{lem}[thm]{Lemma}

\theoremstyle{definition}
\newtheorem{defn}[thm]{Definition}

\theoremstyle{remark}
\newtheorem{rem}[thm]{Remark}
\newtheorem{rems}[thm]{Remarks}

\definecolor{Khalid}{rgb}{0.8, 0.1, 0.1}

\newcommand{\D}{\mathbb{D}}
\renewcommand{\L}{\mathrm{L}}

\newcommand{\Cont}{\mathrm{C}}

\renewcommand{\d}{\,\mathrm{d}}
\newcommand{\dd}{\,\mathrm{d}}
\newcommand{\E}{\mathcal{E}}

\newcommand{\R}{\mathbb{R}}
\newcommand{\C}{\mathbb{C}}
\newcommand{\norm}[1]{\Vert#1\Vert}
\newcommand{\abs}[1]{\vert#1\vert}
\newcommand{\one}{\boldsymbol{1}}
\newcommand{\F}{\mathcal{F}}
\renewcommand{\H}{\mathcal{H}}

\colorlet{shadecolor}{gray!20}
\pgfplotsset{compat=1.9}

\usepgflibrary{fpu}
\makeatletter
\let\c@equation\c@thm
\makeatother
\numberwithin{equation}{section}

\author{Khalid Baadi}
\address{Universit{\'e} Paris-Saclay, CNRS, Laboratoire de Math\'{e}matiques d'Orsay, 91405 Orsay, France}
\email{khalid.baadi@universite-paris-saclay.fr}
\date{\today}
\title[Degenerate Parabolic Riesz Transforms]{A Boundedness Result for Degenerate Parabolic Riesz Transforms with Rough Coefficients}

\begin{document}

\begin{abstract}
In this paper, we prove a boundedness result for parabolic Riesz transforms in the range $p \le 2$ associated with degenerate parabolic operators whose elliptic part is in divergence form with complex coefficients depending measurably on space and time. The degeneracy is determined by a spatial weight in the Muckenhoupt class $A_2(\R^n)$. The argument relies on off-diagonal estimates for the parabolic gradient of the resolvent family, together with weighted embeddings. We obtain boundedness in a range of exponents strictly below $2$, with explicit ranges quantified by the reverse H\"older class of the weight. For real coefficients, this quantification is sharp and extends down to $1$.
\end{abstract}

\maketitle

\setcounter{tocdepth}{1}
\tableofcontents

\section{Introduction}

In the variables $(x,t) \in \R^n \times \R$, we consider degenerate parabolic operators of the form
\begin{equation*}
    \H :=  \partial_t  -\omega^{-1} \mathrm{div}_x(A\nabla_x ),
\end{equation*}
where $A = A(x,t)$ is a  matrix-valued function with complex measurable coefficients and the weight $\omega=\omega(x)$ belongs to the Muckenhoupt class $A_2(\R^n)$. The degeneracy is governed by $\omega$, in the sense that $\omega^{-1}A$ satisfies the classical boundedness and accretivity conditions.

We state the main result of this paper below. Precise definitions will be provided in the following sections. For all $p \in  [1,\infty)$, we define the measures $\mu_p$ on $\R^{n+1}$ by
\begin{equation}\label{eq: mu_p}
    \d \mu_p(x,t) := \omega^{\frac{p}{2}}(x) \d x \d t.
\end{equation}
The operator $\H$ can be shown to be maximal accretive on $\L^2_{\mu_2}(\R^{n+1})$, so that one can use its uniformly bounded resolvents on $\L^2_{\mu_2}(\R^{n+1})$:
\begin{equation*}
    \E_\lambda := (1+\lambda^2 \H)^{-1}, \quad \lambda>0.
\end{equation*}
We now introduce the exponent:
\begin{align*}
    p_-(\H) 
    &:= \inf \left\{ p \in ( 1,\infty) : \ (\mathcal{E}_\lambda)_{\lambda > 0} \text{ extends to a uniformly bounded family on } \L^p_{\mu_p}(\R^{n+1})   \right. \\
    &\hspace{4.5cm} \left. \& \ \  \mu_{p'} \text{ is locally finite on } \R^{n+1} \right\}. \notag
\end{align*}
By $p'$ we mean the H\"older conjugate of $p$. Note that $\mu_{p'}$ is locally finite on $\R^{n+1}$ if and only if $\omega^{\frac{p'}{2}}$ is locally integrable on $\R^n$.

\begin{thm}\label{thm: théorème principal}
Let $n \ge 1$ and $\omega \in A_2(\R^n)$. The following assertions hold. 
\begin{enumerate}[label=(\arabic*), leftmargin=*, itemsep=1pt]
    \item \label{Point (1)} \textbf{Range of exponents:} We have $p_-(\H)\in [1,2)$ with an upper bound depending only on $[\omega]_{A_2}$ and $n$. In addition, for every $q \in (2_\star,2)$, with $2_\star=2-\frac{4}{n+4}\in (1,2)$,
    \begin{equation*}
    \omega \in RH_{\frac{q'}{2}}(\R^n)\quad \Longrightarrow \quad p_-(\H) < q .
    \end{equation*}
    In the endpoint case $q = 2_\star$, if $\omega \in RH_{1+\frac{2}{n}}(\R^n)$, then $p_-(\H) \le 2_\star$.
    \item \label{Point (2)} \textbf{Boundedness of the parabolic Riesz transform:} For every $p \in (p_-(\H),2]$, the parabolic Riesz transform
    \begin{equation*}
    \mathcal{R}_{\H} = (\nabla_x \H^{-1/2}, D_t^{1/2}\H^{-1/2})
    \end{equation*}
    extends to a bounded operator on $\L^p_{\mu_p}(\R^{n+1})$.
    \item \label{Point (3)} \textbf{Case of real coefficients:} If $A$ has real coefficients, then for every $q \in (1,2]$,
    \begin{equation*}
        p_-(\H) < q \quad \Longleftrightarrow \quad \omega \in RH_{\frac{q'}{2}}(\R^n).
    \end{equation*}
    In particular, $p_-(\H) = 1$ if and only if
    \begin{equation*}
         \omega \in \bigcap_{q \in (1,2]} RH_{\frac{q'}{2}}(\R^n)=: RH_\star(\R^n).
    \end{equation*}
    Moreover, if $\omega \in RH_{\infty}(\R^n)\subsetneq RH_\star(\R^n)$, then $p_-(\H) = 1$, and the spatial gradient component $\omega^{\frac{1}{2}} \nabla_x \H^{-1/2}$ extends to a bounded operator from $\L^1_{\mu_1}(\R^{n+1})$ into $\L^{1,\infty}(\R^{n+1})$.
\end{enumerate}
\end{thm}
\noindent This is an extrapolation result from the parabolic Kato square root estimate (case $p=2$), proved by Ataei, Egert, and Nystr\"om \cite{AEN2025}. Before outlining the paper and discussing related work, we list examples of weights illustrating our results. See Section~\ref{sec: weights} for details.
\begin{enumerate}[label=$\blacklozenge$]
    \item A weight $\omega \in A_2(\R^n)$ is given by
    \begin{equation*}
    \omega(x) = |x|^\beta, \quad \forall x \in \R^n, \quad \text{with } -n < \beta < n.
    \end{equation*}

     \item A weight $\omega \in A_2(\R^n) \cap RH_{\frac{q'}{2}}(\R^n)$, with $q \in (1,2]$, is given by
    \begin{equation*}
    \omega(x) = |x|^\beta, \quad \forall x \in \R^n, \quad \text{with } -\frac{2n}{q'} < \beta < n.
    \end{equation*}

    \item A weight $\omega \in A_2(\R^n) \cap RH_{1+\frac{2}{n}}(\R^n)$ is given by
    \begin{equation*}
    \omega(x) = |x|^\beta, \quad \forall x \in \R^n, \quad \text{with } -\frac{n^2}{(n+2)} < \beta < n.
    \end{equation*}

    \item A weight $\omega \in A_2(\R^n) \cap RH_{\infty}(\R^n)$ is given by
    \begin{equation*}
    \omega(x) = |x|^\beta, \quad \forall x \in \R^n, \quad \text{with } 0 \le \beta < n.
    \end{equation*}

    \item A weight $\omega \in \left( A_2(\R^n) \cap RH_\star(\R^n)\right)\setminus RH_{\infty}(\R^n)$ is given by
    \begin{equation*}
    \omega(x) = \max\left\{ \ln\left(1/|x|\right),1 \right\}, \quad \forall x \in \R^n.
    \end{equation*}
\end{enumerate}

The backward operator $\H^\star := -\partial_t - \omega^{-1} \mathrm{div}_x\!\left(A^\star \nabla_x\right)$, where $A^\star$ is the Hermitian adjoint of the matrix $A$, is the adjoint of $\H$ and belongs to the same class of parabolic operators. Hence, Theorem~\ref{thm: théorème principal} also applies to $\H^\star$, and Point~\ref{Point (2)} with duality yields the following.

\begin{cor}\label{cor: reverse}
For every $p\in [2,p_-(\H^\star)')$, one has
\begin{equation}\label{eq: reverse}
\norm{\H^{\nicefrac{1}{2}}u}_{\L^p_{\mu_p}(\R^{n+1})} \lesssim \norm{\D u}_{\L^p_{\mu_p}(\R^{n+1})}, \quad \text{for all } u \in \Cont^\infty_0(\R^{n+1}).
\end{equation}
\end{cor}

The proof of Point~\ref{Point (1)} of Theorem~\ref{thm: théorème principal} is given in Section~\ref{s4}, while the case of real-valued coefficients in Point~\ref{Point (3)} is addressed in Section~\ref{s6}. For complex coefficients, our approach relies on weighted parabolic Gagliardo--Nirenberg inequalities, whereas for real coefficients Gaussian bounds yield more precise estimates and wider admissible ranges. The proof of Point~\ref{Point (2)} follows the one developed in the unweighted theory~\cite{BEK26}, once we have the off-diagonal estimates established in Section~\ref{s3}. This is why we only give the main steps in Section~\ref{s5}. In Section~\ref{s7}, we prove Corollary~\ref{cor: reverse}. Finally, in Section~\ref{s8}, we present and comment on several open questions.

Our contribution here is to bring in some weighted embeddings~\cite{auscherbaadi2025hardy} allowing to quantify ranges of exponents by the reverse H\"older properties of $\omega$. Our approach is to view the weight as a multiplication operator. More precisely, if $T$ is an operator on $\L^2_{\mu_2}(\R^{n+1})$, then the conjugated operator $\omega^{\frac{1}{2}} T \omega^{-\frac{1}{2}}$ acts on the unweighted space $\L^2(\R^{n+1})$. This explains the local finiteness assumption on $\mu_{p'}$ in the definition of $p_-(\mathcal{H})$, which is needed for density and duality arguments, as in the proof of off-diagonal estimates.

Let us briefly recall some related literature for context and motivation. In the elliptic case, \textit{i.e.} $\mathcal{L}=-\omega^{-1} \mathrm{div}_x(A(x)\nabla_x)$, the boundedness of the elliptic Riesz transform $\nabla_x \mathcal{L}^{-1/2}$ is well understood, first in the unweighted case $\omega=1$, see \cite{auscher2007necessary}, and then in the degenerate case, with a weighted theory available in each setting. See \cite{cruz2017kato} and the series of papers by P.~Auscher and J.~M.~Martell \cite{AM1,AM3,AM2,AM4}. In the parabolic case, only the parabolic Kato square root estimate is known, first in the unweighted case \cite{AEN2020} and then in the degenerate case \cite{AEN2025}. The case $p \neq 2$ in the unweighted setting appears in~\cite{Ouhabaz21} with autonomous real coefficients $A(t,x)=A(x)$, using interpolation arguments based on maximal regularity for the Cauchy problem associated with $\H$. For the non-autonomous case, in a recent joint work with M.~Egert and B.~W.~Kosmala~\cite{BEK26}, we treat the case $p \le 2$, and, to the best of our knowledge, this is the first result covering this generality. The present paper presents an extension to degenerate operators. We stress that this is not just a routine extension to a weighted setting, as the appearance of $\mu_p$ varying with $p$ suggests.

\subsubsection*{\textbf{Notation:}} Throughout this paper, we adopt the following notation.
\begin{enumerate}[label=$\blacklozenge$, leftmargin=*, itemsep=0pt]
\item We fix an integer $n \ge 1$ and work in $\R^{n+1} = \R^n \times \R$. 
\item For any $p \in [1, \infty]$, $p'$ denotes its H\"older conjugate.
\item With $d=n+2$, we define $2_\star := \frac{2d}{d+2} \in (1,2)$ and $2^\star := \frac{2d}{d-2} > 2$, the lower and upper parabolic Sobolev conjugates of $2$, respectively. Note that $(2_\star)' = 2^\star$ and that $\nicefrac{2^\star}{2}=1+\nicefrac{2}{n}$.
\item For all $1\le p \le q \le \infty$, we set $\gamma_{p,q}:= (n+2) \left(  \nicefrac{1}{q} - \nicefrac{1}{p} \right)  .$
\item For $(x,t),(y,s)\in \R^{n+1}$, we denote by $\d((x,t),(y,s))=\max\big(|x-y|_\infty,\sqrt{|t-s|}\big)$ the parabolic distance between them.
\item For $(x,t)\in\R^{n+1}$ and $r>0$, let $Q_r(x)$ denote the cube centered at $x$ with radius $r$ and sides parallel to the coordinate axes, and set $I_r(t)\coloneqq (t-r^2,t+r^2)$. The parabolic cube centered at $(x,t)$ with radius $r$ is $\Delta_r(x,t)\coloneqq Q_r(x)\times I_r(t)$. We usually omit the center.
\item For a parabolic cube $\Delta_r:=Q_r\times I_r \subset \R^{n+1}$ and $a,b>0$, we define the stretched cube $aQ_r\times bI_r:=Q_{ar}\times I_{br}$. For $N>1$, we set $C^N_1(\Delta_r):=4Q_r \times N^2 I_r$ and for all $j\ge 2$
\begin{equation*}
   C^N_j(\Delta_r):= \left ( 2^{j+1}Q_r \times N^{j+1} I_r \right ) \setminus \left ( 2^{j}Q_r \times N^{j} I_r \right ).
\end{equation*}
\item In the above three points, the $|\cdot|_\infty$ norm can be replaced by the Euclidean norm on $\R^n$, and cubes by Euclidean balls. Since the parabolic distances are equivalent and the weights are doubling, this does not affect our arguments.
\item For $ p \in [1, \infty]$, $\L^p(\R^{n+1})$ is always with respect to the Lebesgue measure; otherwise, the underlying measure will be specified as a subscript. We write $\|\cdot\|_p$ for its norm.
\item The notation $C = C(a,b,\dots)$ for a constant means that $ C \in (0, \infty)$ and depends only on $(a,b,\dots)$. The notation $\lesssim_p$ means that the quantity is bounded by a constant depending on the assumed $\L^p$-boundedness, eventually after interpolation with $\L^2$ bounds.
\item We abbreviate the term ``locally finite'' as ``l.f.''.
\end{enumerate}

\subsubsection*{\textbf{Acknowledgements:}}

The author thanks his PhD advisor, Professor Pascal Auscher, for stimulating discussions and useful suggestions improving earlier versions of this manuscript.

\section{Preliminaries and basic assumptions}

We begin by specifying the main assumptions on the weights, the underlying functional spaces, and the tools that will be used repeatedly in the sequel.

\subsection{Weight classes}\label{sec: weights}

We refer to \cite[Ch.~V]{Stein1993_HA} and \cite[Ch.~7]{grafakos2008classical} for classical facts on weights. All cubes $Q \subset \R^n$ appearing in the forthcoming definitions are assumed to have sides parallel to the coordinate axes, and we shall not mention this assumption henceforth. We say that an almost everywhere (for Lebesgue measure) defined function $\omega : \R^n \rightarrow [0,\infty]$ is a weight if it is locally integrable on $\R^n$. The Muckenhoupt class $A_2(\R^n)$ is defined as the set of all weights satisfying
\begin{equation}\label{MuckWeight}
      [ \omega  ]_{A_2}:= \sup_{Q \subset \R^n  } \left ( \frac{1}{|Q|} \int_Q \omega(x) \, \mathrm{d}x   \right ) \left ( \frac{1}{|Q|} \int_Q \omega^{-1}(x) \, \mathrm{d}x   \right ) < \infty,
\end{equation}
where the supremum is over all cubes $Q \subset \R^n$, and $|Q|$ is their Lebesgue measure. The reverse H\"older class $RH_q(\R^n)$, $q\in [1,\infty)$, consists of all weights $\omega$ such that, for some constant $C$,
\begin{equation*}
    \left ( \frac{1}{|Q|} \int_Q \omega^q(x) \, \mathrm{d}x   \right )^{\frac{1}{q}} \leq  \frac{C}{|Q|} \int_Q \omega(x) \, \mathrm{d}x,
\end{equation*}
for all cubes $Q \subset \R^n$. For $q=\infty$, the reverse H\"older class $RH_\infty(\R^n)$ consists of all weights $\omega$ such that there exists a constant $C$ such that, for all cubes $Q \subset \R^n$,
\begin{equation*}
    \omega(y) \leq \frac{C}{|Q|} \int_Q \omega(x)\, \mathrm{d}x, \quad \text{for almost every } y \in Q.
\end{equation*}
The infimum of such constants $C$ is denoted by $[\omega]_{RH_q}$. Finally, we set
\begin{equation*}
   RH_\star(\R^n)=\bigcap_{s \in [1,\infty)} RH_{s}(\R^n)= \bigcap_{q \in (1,2]} RH_{\frac{q'}{2}}(\R^n).
\end{equation*}
In the above definitions, cubes may be replaced by Euclidean balls. For simplicity, we keep the same notation and constants. We recall below some known facts about these classes.
\begin{prop}\label{prop:weights}
\begin{enumerate}[label=(\arabic*), leftmargin=*, itemsep=0pt]
    \item $RH_q(\R^n) \subset RH_p(\R^n)$ for all $1 \le p \leq q \le \infty$.
    \item If $\omega \in RH_q(\R^n)$ with $1<q<\infty$, then there exists $p \in (q,\infty)$ such that $\omega \in RH_p(\R^n)$.
    \item If $\omega \in A_2(\R^n)$, then there exists $\varepsilon=\varepsilon([\omega]_{A_2},n)$ such that $\omega \in RH_{1+\varepsilon}(\R^n)$.
\end{enumerate}
\end{prop}

For the first three examples of weights given in the introduction, see \cite[Remark~2.3]{auscherbaadi2025hardy}, which explains how to obtain them from the literature. For the fourth one, by \cite[Theorem~4.4]{MR1308005}, $\omega \in A_2(\R^n) \cap RH_{\infty}(\R^n)$ if and only if $\omega^{-1}\in A_1(\R^n)$; see \cite[Example~7.1.7]{grafakos2008classical} for examples of the latter class. Finally, the example $\omega \in \bigl( A_2(\R^n) \cap RH_\star(\R^n)\bigr)\setminus RH_{\infty}(\R^n)$ is taken from \cite[Sect.~4]{MR1308005}, and one easily verifies that it satisfies the claim.

From now on and \textbf{throughout this paper, $\omega$ denotes a fixed weight belonging to the Muckenhoupt class $A_2(\R^n)$}. For every $p\in  [1,\infty)$ and $K \subset \R^n$ a measurable set, we denote by $\L^p_\omega(K)$ the space of all measurable functions $f:K \rightarrow \mathbb{C}$ such that $$\|f\|_{\L^p_\omega(K)}:=\left ( \int_K |f|^p \, \mathrm d \omega  \right )^{1/p}<\infty.$$
In particular, $\L^2_\omega(\R^n)$ is the Hilbert space of square-integrable functions on $\R^n$ with respect to $\mathrm{d}\omega$. We denote its norm by ${\lVert \cdot \rVert}_{2,\omega}$ and its inner product by $\langle \cdot , \cdot \rangle_{2,\omega}$.

For all $p \in  [1,\infty)$, we define $\mu_p$ on $\R^{n+1}$ by \eqref{eq: mu_p}. For $p=2$, we simply write $\d \mu$ instead of $\d \mu_2$. The same notation for $\omega$ applies when it is replaced by any weight on $\R^n$ or $\R^{n+1}$.

\subsection{Parabolic energy space} 
We denote by $\F$ the Fourier transform with respect to the time variable. Recall that if $u \in \L_\mu^2(\R^{n+1})$, then $u(x,\cdot) \in \L^2(\R)$ for almost every $x \in \R^n$ by Fubini's theorem. Equivalently, again by Fubini's theorem, one may identify $\L_\mu^2(\R^{n+1})$ with $\L^2(\R; \L_\omega^2(\R^n))$ and apply the $\L_\omega^2(\R^n)$-valued Fourier transform, which defines an isomorphism; see \cite{hytonen2016analysis} for details. Then, the expression
\begin{equation*}
    H_t u := \F^{-1}\left(i\,\tfrac{\tau}{|\tau|}\,\F u\right)
\end{equation*}
defines the Hilbert transform. If $|\tau|^{1/2}\F u \in \L_\mu^2(\R^{n+1})$, the half-order time derivative is defined by
\begin{equation*}
    D_t^{1/2} u := \F^{-1}\left(|\tau|^{1/2}\F u\right).
\end{equation*}
We now define the parabolic energy space as 
\begin{equation*}
    \mathrm{E}_\mu:= \left\{u\in \L_\mu^2(\R^{n+1}) \, : \, \nabla_xu, \, D_t^{1/2}u \in \L_\mu^2(\R^{n+1})  \right\},
\end{equation*}
where $\nabla_x$ denotes the gradient with respect to the spatial variables $x$, understood in the sense of distributions on $\R^{n+1}$, \textit{i.e.} as an element of $\mathcal{D}'(\R^{n+1})^n$, since $\L^2_{\mu}(\R^{n+1}) \subset \L^1_{\text{loc}}(\R^{n+1})$ by \eqref{MuckWeight}. For $u \in \mathrm{E}_\mu$, its parabolic gradient is defined by $\D u := (\nabla_x u, D_t^{1/2} u)$. We equip $\mathrm{E}_\mu$ with the norm $\|u\|_{\mathrm{E}_\mu} := (\|u\|_{2,\omega}^2 + \|\nabla_x u\|_{2,\omega}^2 + \|D_t^{1/2} u\|_{2,\omega}^2)^{1/2}$, which makes $\mathrm{E}_\mu$ a Hilbert space containing $\Cont_0^\infty(\R^{n+1})$ as a dense subspace. Moreover, multiplication by functions in $C_b^1(\R^{n+1})$ defines a bounded operator on $\mathrm{E}_\mu$. For details, see \cite[Lemma 3.3]{AEN2025}. In particular, the triple $(\mathrm{E}_\mu, \L_\mu^2(\R^{n+1}), \mathrm{E}_\mu^\star)$ forms a Gelfand triple, where $\mathrm{E}_\mu^\star$ is the anti-dual of $\mathrm{E}_\mu$. Moreover, one defines the bounded operators
\begin{align*}
    \nabla_x  : \, \mathrm{E}_\mu \longrightarrow \L_\mu^2(\R^{n+1})^n, \quad D_t^{1/2}  : \, \mathrm{E}_\mu \longrightarrow \L_\mu^2(\R^{n+1}), \quad
    \D   : \, \mathrm{E}_\mu \longrightarrow \L_\mu^2(\R^{n+1})^{n+1},
\end{align*}
and their adjoints
\begin{align*}
    -\omega^{-1}\mathrm{div}_x \, \omega : \, \L_\mu^2(\R^{n+1})^n \longrightarrow \mathrm{E}_\mu^\star, \quad
    D_t^{1/2} : \, \L_\mu^2(\R^{n+1}) \longrightarrow \mathrm{E}_\mu^\star.
\end{align*}
We finally recall the following representation formula for $D_t^{1/2}$: for $u \in \Cont_0^\infty(\R^{n+1})$, one has
\begin{equation}\label{eq: Dt representation}
    D_t^{1/2} u(x,t) = \frac{1}{2\sqrt{2\pi}} \int_{\R} \frac{u(x,t) - u(x,s)}{|t-s|^{3/2}} \, \mathrm{d}s, \quad \text{for almost every } (x,t) \in \R^{n+1}.
\end{equation}
See \cite[§12.1]{Kilbas} for details. Thus, by a density argument, for $u \in \mathrm{E}_\mu$, for almost every $x \in \R^n$ and almost every $t\in \R \setminus \mathrm{supp}(u(x,\cdot))$, one has
\begin{equation*}
    D_t^{1/2} u(x,t) = -\frac{1}{2\sqrt{2\pi}} \int_{\R} \frac{u(x,s)}{|t-s|^{3/2}} \, \mathrm{d}s.
\end{equation*}

\subsection{Toolbox for \texorpdfstring{$\L^p-\L^q$}{} bounded families}

We recall several abstract principles that will be used in the sequel, concerning (unweighted) $\L^p-\L^q$ boundedness for uniformly bounded families $(T_\lambda)_{\lambda \in \mathcal{U}}$ of operators defined on $\L^2(\R^{n+1})$, indexed by a set $\mathcal{U} \subset \mathbb{C} \setminus \{0\}$. In our setting, these families will consist of powers of the resolvent family or of their parabolic gradients associated with the degenerate parabolic operator, conjugated with $\omega^{\frac{1}{2}}$ and $\omega^{-\frac{1}{2}}$.

\begin{defn}\label{def: toolbox}
Let $1\le p \le q \le \infty$. The family $(T_\lambda)_{\lambda \in \mathcal{U}}$ is said to 
\begin{enumerate}
    \item be  $\L^p-\L^q$ bounded if there exists a constant $C>0$ such that 
    \begin{equation*}
        \| T_\lambda u\|_q \le C |\lambda|^{\gamma_{p,q}} \|  u\|_p,
    \end{equation*}
    for all $\lambda\in \mathcal{U}$ and all $u \in \L^p(\R^{n+1}) \cap \L^2(\R^{n+1})$.
    \item satisfy the  $\L^p-\L^q$ off-diagonal estimates if there exist constants $C,c$ such that
    \begin{equation*}
        \| \one_F T_\lambda (\one_E u)\|_q \le C |\lambda|^{\gamma_{p,q}} e^{-c \frac{\d(E,F)}{|\lambda|}} \| \one_E u\|_p,
    \end{equation*}
    for all $\lambda\in \mathcal{U}$, $u \in \L^p(\R^{n+1}) \cap \L^2(\R^{n+1})$ and all measurable sets $E,F\subset \R^{n+1}$, where $\d(E,F)$ is the parabolic distance between $E$ and $F$.
\end{enumerate}
When $p=q$, we speak of  $\L^p$-boundedness and  $\L^p$ off-diagonal estimates, respectively.
\end{defn}

We recall the following useful results. The proofs can be obtained by adapting those in \cite[Chapter~4]{AEbook2023} to the parabolic scaling.

\begin{lem}\label{lem: Toolbox}
Let $(T_\lambda)_{\lambda \in \mathcal{U}}$ and $(S_\lambda)_{\lambda \in \mathcal{U}}$ be operator families as above. Let $1\le p \le q \le r \le \infty$.
\begin{enumerate}
    \item \emph{(Duality)} $(T_\lambda)_{\lambda \in \mathcal{U}}$ is  $\L^p-\L^q$ bounded if and only if $(T^\star_\lambda)_{\lambda \in \mathcal{U}}$ is  $\L^{q'}-\L^{p'}$ bounded.
    \item \emph{(Composition)} If $(T_\lambda)_{\lambda \in \mathcal{U}}$ is  $\L^p-\L^q$ bounded and $(S_\lambda)_{\lambda \in \mathcal{U}}$ is  $\L^q-\L^r$ bounded, then $(S_\lambda T_\lambda)_{\lambda \in \mathcal{U}}$ is  $\L^p-\L^r$ bounded. 
    \item \emph{(Extrapolation)} If $(T_\lambda)_{\lambda \in \mathcal{U}}$ satisfies  $\L^p-\L^q$ off-diagonal estimates, then it is  $\L^p$-bounded and  $\L^q$-bounded.
\end{enumerate}
Points (1) and (2) remain valid if boundedness is replaced by off-diagonal estimates.
\end{lem}

We also recall the following useful bootstrapping argument, which follows \emph{verbatim} from the proof of \cite[Lemma~4.4]{AEbook2023}, where the case $q=2$ is treated.
\begin{lem}[Triangle interpolation]\label{lem: gives m}
Let $q\in (1,\infty]$ and $(T_\lambda)_{\lambda \in \mathcal{U}}$ be an  $\L^q$-bounded family. Assume that there exist $p,\varrho \in [1,q)$ such that $(T_\lambda)_{\lambda \in \mathcal{U}}$ is  $\L^p$-bounded and  $\L^\varrho-\L^q$ bounded. Then, for all $r\in (p,q]$, there exists an integer $m \ge 1$ such that $(T^m_\lambda)_{\lambda \in \mathcal{U}}$ is  $\L^r-\L^q$ bounded. 
\end{lem}

\section{The parabolic operator \texorpdfstring{$\H$}{}: the \texorpdfstring{$\L^2$}{}-theory and off-diagonal estimates}\label{s3}

In this section, we define the parabolic operator $\H$, recall known results, and explain how to prove off-diagonal estimates for the family $(\lambda \D \E_\lambda)_{\lambda>0}$, a key tool for extrapolation.

\subsection{The parabolic operator \texorpdfstring{$\H$}{} and its parabolic Riesz transform}

We fix a matrix-valued function $A: \R^{n+1} \rightarrow M_n(\mathbb{C})$ with complex measurable coefficients, satisfying
\begin{equation}\label{eq: ellipticité A}
\left| A(x,t) \xi \cdot \zeta \right| \leq M |\xi|\,|\zeta|, \quad 
\nu |\xi|^2 \leq \mathrm{Re}(A(x,t) \xi \cdot \overline{\xi})
\end{equation}
for some $M, \nu > 0$ and for all $\xi, \zeta \in \C^n$ and $(x,t) \in \R^{n+1}$.

We may now define the parabolic operator $\H: \mathrm{E}_\mu \rightarrow \mathrm{E}_\mu^\star$ by setting, for all $u,v \in \mathrm{E}_\mu$,
\begin{equation*}
    \H(u)(v) = \iint_{\R^{n+1}} H_t D_t^{1/2} u(x,t) \cdot \overline{D_t^{1/2} v(x,t)} + \omega^{-1}(x)A(x,t) \nabla_x u(x,t) \cdot \overline{\nabla_x v(x,t)} \, \, \mathrm{d}\mu(x,t).
\end{equation*}
We write formally
\begin{equation*}
    \H = \partial_t - \omega^{-1} \mathrm{div}_x(A \nabla_x).
\end{equation*}
Likewise, we define $\H^\star : \mathrm{E}_\mu \to \mathrm{E}_\mu^\star$ as the adjoint of $\H$ associated with the adjoint form defining $\H$. Formally, it is given by $\H^\star := -\partial_t -  \omega^{-1} \mathrm{div}_x\!\left(A^\star \nabla_x\right)$, where $A^\star$ denotes the Hermitian adjoint of the matrix $A$.

Our starting point is the Kato square root estimate for $\H$ \cite{AEN2025}, stated as follows.
\begin{thm}\label{thm: théorie L2}
The maximal restriction of $\H$ to $\L_\mu^2(\R^{n+1})$ is an injective maximal accretive operator on $\L_\mu^2(\R^{n+1})$, and the domain of its unique maximal accretive square root is $\mathrm{E}_\mu$, that is, $\mathrm{dom}(\sqrt{\H}) = \mathrm{E}_\mu$. Moreover, there exists a constant $C=C(M,\nu,[\omega]_{A_2},n)\ge 1$ such that
\begin{equation}
  \frac{1}{C} \|\D u\|_{2,\mu} \le \|\sqrt{\H} u\|_{2,\mu} \le C \|\D u\|_{2,\mu}, \quad \text{for all } u \in \mathrm{E}_\mu.
\end{equation}
\end{thm}

This result shows that the operator $\mathcal{R}_{\H} := \D\H^{-1/2}$, initially defined from $\operatorname{ran}(\sqrt{\H})$ into $\L_\mu^2(\R^{n+1})^{n+1}$, is of strong type $(2,2)$. Since $\sqrt{\H}$ is maximal accretive and injective on $\L_\mu^2(\R^{n+1})$, its range is dense in $\L_\mu^2(\R^{n+1})$, see \cite[Proposition 2.4]{egert2024harmonic}. Therefore, $\mathcal{R}_{\H}$ extends uniquely to a bounded operator from $\L_\mu^2(\R^{n+1})$ to $\L_\mu^2(\R^{n+1})^{n+1}$, still denoted by $\mathcal{R}_{\H}$. It is this extension that we call \textbf{the parabolic Riesz transform associated with $\H$}. To simplify the notation, we will abusively write $\L_\mu^2(\R^{n+1})$ instead of $\L_\mu^2(\R^{n+1})^{n+1}$.

\subsection{Off-diagonal estimates for \texorpdfstring{$( \mathcal{E}_\lambda)_{\lambda>0}$}{} and \texorpdfstring{$(\lambda \nabla_x \mathcal{E}_\lambda)_{\lambda>0}$}{} families}
In the sequel, we use the following notation for the resolvent family of $\H$ and its adjoint $\H^\star$. For $\lambda > 0$, we set
\begin{equation}
    \E_\lambda := (1+\lambda^2 \H)^{-1}, \quad \E_\lambda^\star := (1+\lambda^2 \H^\star)^{-1}.
\end{equation}
They are uniformly bounded adjoint families on $\L_\mu^2(\R^{n+1})$ by maximal accretivity of $\H$. We stress that the adjointess is with respect to $\langle \cdot, \cdot \rangle_{2,\mu}$. We also recall the following key result on off-diagonal estimates; for a proof, see \cite[Lemma~4.3 and Lemma~4.4]{AEN2025}.
\begin{prop}[Classical off-diagonal estimates]\label{prop: classical ODEs}
The two families $(\omega^{\frac{1}{2}} \E_\lambda \, \omega^{-\frac{1}{2}})_{\lambda>0}$ and $(\omega^{\frac{1}{2}} \lambda \nabla_x \E_\lambda \, \omega^{-\frac{1}{2}})_{\lambda>0}$ satisfy $\L^2$ off-diagonal estimates and the family $(\omega^{\frac{1}{2}} \lambda \D \E_\lambda \, \omega^{-\frac{1}{2}})_{\lambda>0}$ is $\L^2$-bounded, with constants $C$ and $c$ depending only on $M$, $\nu$, and $n$.
\end{prop}

\subsection{Off-diagonal estimates for \texorpdfstring{$(\lambda  \D \mathcal{E}_\lambda)_{\lambda>0}$}{} family}

\subsubsection{Off-diagonal estimates for the \texorpdfstring{$(\lambda D_t^{1/2} \mathcal{E}_\lambda)_{\lambda>0}$}{} family}

We begin with the spatial supports, for which we have the following result.

\begin{prop}[Off-diagonal estimates on spatial supports]\label{prop: L2 spatial supports}
There exist two constants $C = C(M,\nu,n)$ and $c = c(M,\nu,n)$ such that the following off-diagonal estimate hold:
\begin{equation}\label{eq: Dt space off}
    \|\one_{F\times\R} \, \omega^{\frac{1}{2}}  \lambda D_t^{1/2} \mathcal{E}_\lambda (\omega^{-\frac{1}{2}} \one_{E\times\R} u) \|_{2} \le C e^{-c\frac{\d(E,F)}{\lambda}} \|\one_{E\times\R} u\|_{2},
\end{equation}
for all $\lambda>0$, measurable sets $E,F\subset \R^n$ and $u\in \L^2(\R^{n+1})$, where $\d(E,F)$ denotes the euclidean distance between $E$ and $F$ in $\R^n$.
\end{prop}
\begin{proof}
Using the variational equation and proceeding as in \cite[Proposition 4.1]{BEK26}, we prove that there exist two constants $C = C(M,\nu,n)$ and $c = c(M,\nu,n)$ such that     
 \begin{equation*}
    \|\one_{F\times\R} \lambda D_t^{1/2} \mathcal{E}_\lambda (\one_{E\times\R} u) \|_{2,\mu} \le C e^{-c\frac{\d(E,F)}{\lambda}} \|\one_{E\times\R} u\|_{2,\mu},
\end{equation*}
for all $\lambda>0$ and $u\in \L_\mu^2(\R^{n+1})$. This is equivalent to \eqref{eq: Dt space off}.
\end{proof}

For the temporal supports, we have the following result.

\begin{prop}[Off-diagonal estimates on temporal supports]\label{prop: time supports}
Let $p \in (1,2]$ be such that the families $(\omega^{\frac{1}{2}} \, \mathcal{E}_\lambda \, \omega^{-\frac{1}{2}})_{\lambda>0}$ and $(\omega^{\frac{1}{2}} \, \lambda D_t^{1/2}\mathcal{E}_\lambda \, \omega^{-\frac{1}{2}})_{\lambda>0}$ are $\L^p$-bounded, and that $\mu_{p'}$ is l.f. on $\R^{n+1}$. For every $N>1$, there exists a constant $C = C(N,p)$  such that the following off-diagonal estimates hold:
\begin{equation*}
    \|\one_{\R^n\times F} \, \omega^{\frac{1}{2}} \lambda D_t^{1/2}\mathcal{E}_\lambda \left ( \omega^{-\frac{1}{2}}  \one_{\R^n\times E} u \right ) \|_{p}
        \lesssim_p C\left(\frac{\lambda}{r} + \left(\frac{\lambda}{r}\right)^2\right)
        N^{-j\varepsilon}
        \|\one_{\R^n\times E} u\|_{p},
\end{equation*}
for all $\lambda, r > 0$, all $j \ge 1$ and all $u \in \L^p(\R^{n+1})\cap \L^2(\R^{n+1})$ in the following two scenarios:
\begin{enumerate}[label=(\roman*), ref=\roman*]
    \item\label{item: cube to annulus time} \emph{(From the cube to the annulus)} For $E= N^{j\delta}I_r$, $F= N^{j+1}I_{r} \setminus N^{j} I_{r}$ with $\varepsilon = 1+\frac{1}{1+p'}$. Here, the parameter $\delta \in [0,\gamma)$, where \begin{equation}\label{eq: epsilon*}
    \gamma= \frac{2+p'}{2+2p'} \in \left ( \frac{1}{2} , \frac{3}{4} \right ),
    \end{equation}
    
    \item\label{item: annulus to cube time} \emph{(From the annulus to the cube)} 
    For $E= N^{j+1}I_r \setminus N^j I_r$, $F= I_r$ with $\varepsilon = 2$.
\end{enumerate}

\end{prop}
\begin{proof}
For $\eta \in \Cont^\infty_b(\R)$ and $u\in \L_\mu^2(\R^{n+1})$, we have for all $\lambda>0$:
\begin{align*}
    \eta \lambda \E^\star_\lambda D_t^{1/2}u=\lambda [\eta, \E_\lambda^\star]D_t^{1/2}u+\lambda \E^\star_\lambda \eta D_t^{1/2}u
    &= \lambda \E^\star_\lambda [1+\lambda^2 \H^\star, \eta]\E_\lambda^\star D_t^{1/2}u+\lambda \E^\star_\lambda \eta D_t^{1/2}u
    \\&= -\lambda^3 \E^\star_\lambda (\partial_t \eta) \E^\star_\lambda D_t^{1/2}u+\lambda \E^\star_\lambda \eta D_t^{1/2}u
    \\&=: \mathrm{I}^\eta_u+\mathrm{II}^\eta_u,
\end{align*}
where we used that $[A,B^{-1}]=B^{-1}[B,A]B^{-1}$ in the second equality, and that 
\begin{equation}\label{eq: commutator}
    [1+\lambda^2 \H^\star, \eta]=\lambda^2[\H^\star, \eta]=-\lambda^2\partial_t \eta \quad \text{on} \ \mathrm{E}_\mu \ (\text{valued in} \ \mathrm{E}^\star_\mu)
\end{equation}
in the third equality, since $\H^\star = -\partial_t -  \omega^{-1} \mathrm{div}_x\!\left(A^\star \nabla_x\right)$ and $\H^\star$ is local in its variables. For \eqref{eq: commutator}, one first proves the identity on $\Cont_0^\infty(\R^{n+1})$ and then extends it to $\mathrm{E}_\mu$ by density. Now, by duality, the families $(\omega^{\frac{1}{2}}\E^\star_\lambda)_{\lambda>0}$ and $(\omega^{\frac{1}{2}}\lambda \E^\star_\lambda D_t^{1/2})_{\lambda>0}$ are uniformly bounded from $\L^{p'}_{\mu_{p'}}(\R^{n+1}) \cap \L^{2}_{\mu}(\R^{n+1})$ to $\L^{p'}(\R^{n+1})$. Thus, for $u\in \L^{p'}_{\mu_{p'}}(\R^{n+1}) \cap \mathrm{E}_\mu$, we have
\begin{equation*}
    \norm{\omega^{\frac{1}{2}}\eta \lambda \E^\star_\lambda D_t^{1/2}u}_{p'} \le \norm{\omega^{\frac{1}{2}}\mathrm{I}^\eta_{u}}_{p'}+\norm{\omega^{\frac{1}{2}}\mathrm{II}^\eta_{u}}_{p'} \lesssim_p  \lambda^2 \norm{\partial_t \eta}_{\infty} \norm{\omega^{\frac{1}{2}}u}_{p'}+\lambda \norm{\omega^{\frac{1}{2}}\eta D_t^{1/2}u}_{p'}.
\end{equation*}
Now, one proceeds exactly as in \cite[Prop.~4.4 and Rem.~4.9]{BEK26}: one first takes $u \in \Cont_0^\infty(\R^{n+1})$, uses the representation formula \eqref{eq: Dt representation} for $D_t^{1/2}u$, chooses the appropriate $\eta$ for each scenario, and concludes by density. We stress that we assumed that $\mu_{p'}$ is l.f. on $\R^{n+1}$, so that $\Cont_0^\infty(\R^{n+1}) \subset \L^{p'}_{\mu_{p'}}(\R^{n+1}) \cap \L^{2}_{\mu}(\R^{n+1})$ and is a dense subspace of $\L^{p'}_{\mu_{p'}}(\R^{n+1})$. This is also used in the duality argument mentioned earlier.
\end{proof}

\subsubsection{Full gradient off-diagonal estimates and composition}
We state the following theorem.

\begin{thm}\label{thm: parabolicODEs}
Let $\varrho \in [1,2]$ such that the families $(\omega^{\frac{1}{2}} \,\mathcal{E}_\lambda \, \omega^{-\frac{1}{2}})_{\lambda>0}$ and $(\omega^{\frac{1}{2}} \, \lambda \D\mathcal{E}_\lambda \, \omega^{-\frac{1}{2}})_{\lambda>0}$ are $\L^\varrho$-bounded, and that $\mu_{\varrho'}$ is l.f. on $\R^{n+1}$. Let $p\in(\varrho,2]$ or $p=\varrho=2$.
Fix $N\ge 4$ and an integer $m \ge 1$. Then there exists a constant $C$ such that the off-diagonal estimates
\begin{align*}
    \|\one_F \, \omega^{\frac{1}{2}} \lambda \D\mathcal{E}^m_\lambda \left( \omega^{-\frac{1}{2}} \one_E u \right) \|_{p}
        \lesssim_\varrho C\biggl(\frac{\lambda}{r} +\Bigl(\frac{\lambda}{r}\Bigr)^{4N}\biggr)
        N^{-j \varepsilon}
        \|\one_E u\|_{p}
\end{align*}
holds for all $\lambda, r >0$, all $j \ge 2$ and all $u \in \L^p(\R^{n+1}) \cap \L^2(\R^{n+1})$ in the following scenarios:
\begin{enumerate}[leftmargin=*, itemsep=0pt]
    \item[(i)] \emph{(From the cube to the annulus)} For $E= \Delta_r$, $F= C_j^N(\Delta_r)$ with $\varepsilon = 1+\frac{1}{1+p'}$.

    \item[(i)]bis \emph{As in} (i), \emph{but with $E = 2^{j-1}Q_r \times N^{j\delta}I_r$, where $\delta \in [0,\nicefrac{\gamma}{2}]$ and $\gamma$ is defined in \eqref{eq: epsilon*}.}
    
    \item[(ii)] \emph{(From the annulus to the cube)} 
    For $E= C_j^N(\Delta_r)$, $F= \Delta_r$ with $\varepsilon = 2$.
\end{enumerate}
\end{thm}
The proof is exactly the same as that of~\cite[Theorem 4.10]{BEK26}. We combine Propositions~\ref{prop: classical ODEs}, \ref{prop: L2 spatial supports}, and \ref{prop: time supports} by decomposing the supports into spatial and temporal components and estimating them separately. For (i) bis, it is the same argument as for (i); see~\cite[Remark 4.14]{BEK26}. See the figure below for the first case. We omit the details.

\begin{center}
\scalebox{0.75}{%
\begin{tikzpicture}
    \pgfdeclarelayer{layer1}
    \pgfdeclarelayer{layer2}
    \pgfdeclarelayer{layer3}
    \pgfdeclarelayer{layer4}
    \pgfdeclarelayer{layer5}
    \pgfsetlayers{layer1,layer2,layer3,layer4,layer5}

    \tikzmath{\Ballx=2; \Ballt=2; \InnerAnnulusx=3; \InnerAnnulust=4; \OuterAnnulusx=5; \OuterAnnulust=6;}

    \begin{pgfonlayer}{layer1}
        \filldraw[very thick, pattern={Lines[angle=45,distance=5pt]}, pattern color=violet]
        (-\OuterAnnulusx,-\OuterAnnulust)
        rectangle (\OuterAnnulusx,\OuterAnnulust);
    \end{pgfonlayer}

    \begin{pgfonlayer}{layer2}
        \fill[fill=white]
        ({-\OuterAnnulusx+0.02},-\InnerAnnulust)
        rectangle (\OuterAnnulusx-0.02,\InnerAnnulust);
    \end{pgfonlayer}

    \begin{pgfonlayer}{layer3}
        \fill[pattern={Lines[angle=-45,distance=5pt]}, pattern color=teal]
        (-\OuterAnnulusx,-\InnerAnnulust)
        rectangle (\OuterAnnulusx,\InnerAnnulust);
    \end{pgfonlayer}

    \begin{pgfonlayer}{layer4}
        \filldraw[very thick, fill=white]
        (-\InnerAnnulusx,-\InnerAnnulust)
        rectangle (\InnerAnnulusx,\InnerAnnulust);
    \end{pgfonlayer}

    \begin{pgfonlayer}{layer5}
        \filldraw[very thick, fill=lightgray]
        (-\Ballx,-\Ballt)
        rectangle (\Ballx,\Ballt);

        \draw[black, very thick, ->]
        ({-\OuterAnnulusx-0.5},0) -- ({\OuterAnnulusx+0.5},0)
        node[above] {$x$};

        \draw[black, very thick, ->]
        (0,{-\OuterAnnulust-0.5}) -- (0,{\OuterAnnulust+0.5})
        node[left] {$t$};

        \path ({\OuterAnnulusx+1},\OuterAnnulust) node {$C_j^N(\Delta_r)$}
              (-0.5,1) node {$\Delta_r$};
    \end{pgfonlayer}

\end{tikzpicture}%
}
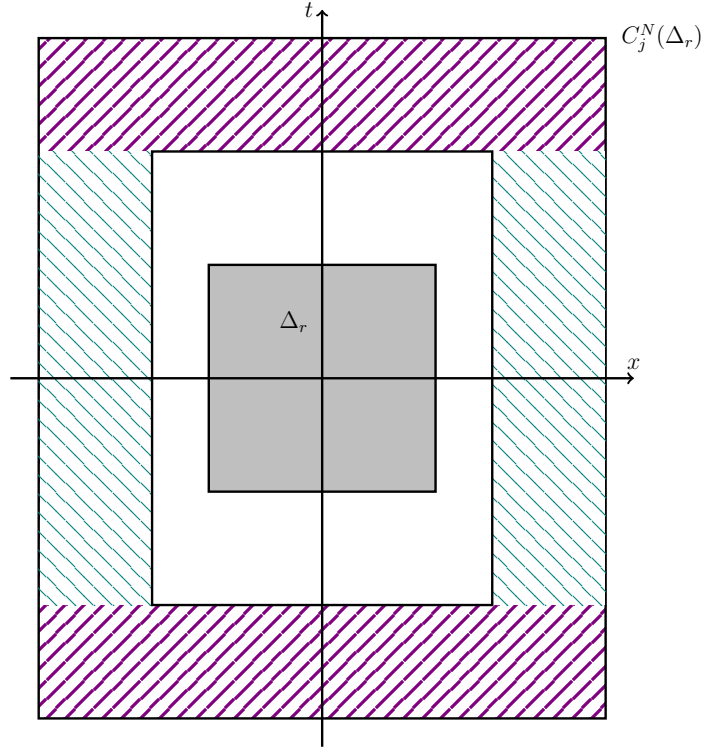
\captionof*{figure}{Spatial support (\textcolor{teal}{teal}) vs.\ time support (\textcolor{violet}{violet})}
\end{center}

\section{Gagliardo--Nirenberg inequalities: Proof of \ref{Point (1)} of Theorem \ref{thm: théorème principal}}\label{s4}

In this section, we first establish Gagliardo--Nirenberg inequalities. They will be the key ingredient in the proof of Point \ref{Point (1)} of Theorem \ref{thm: théorème principal}, which we present next.

\subsection{\texorpdfstring{$\L_\mu^2$}{}-parabolic Sobolev embeddings and Gagliardo--Nirenberg inequalities}
\begin{prop}\label{lem: Sobolev}
Let $q \in [2_\star,2)$ and assume that $\omega \in RH_{\frac{q'}{2}}(\R^n)$. Then
\begin{align*}
    \mathrm{E}_\mu \subset \L_{\mu_{q'}}^{q'}(\R^{n+1}).
\end{align*}
Moreover, there exists a constant $C = C([\omega]_{A_2}, [\omega]_{RH_{\frac{q'}{2}}},q, n)$ such that, for every $u \in \mathrm{E}_\mu$,
\begin{equation}\label{eq: Gagliardo-Nirenberg}
    \norm{u}_{\L_{\mu_{q'}}^{q'}(\R^{n+1})}
    \le C \|\D u\|_{2,\mu}^{\kappa}\|u\|_{2,\mu}^{1-\kappa},
\end{equation}
where $\kappa=\nicefrac{(n+2)}{2}-\nicefrac{(n+2)}{q'}=-\gamma_{2,q'}$. Observe that when $q=2_\star$, or equivalently $q'=2^\star$, one has $\kappa=1$. In this case, the embedding is homogeneous and therefore optimal.
\end{prop}
\begin{proof}
Fix $u \in \mathrm{E}_\mu$. Set $\alpha = n\left(\frac{1}{2} - \frac{1}{q'}\right)$. We define the degenerate Laplacian $-\Delta_\omega$ via the method of forms. We refer to \cite{auscherbaadi2025hardy} for details. By \cite[Theorem 1.1]{auscherbaadi2025hardy}, there exists a constant $C_1 = C_1([\omega]_{A_2}, [\omega]_{RH_{\frac{q'}{2}}}, q, n)$ such that we have
\begin{equation}\label{eq: HLS}
    \norm{u(t)}_{\L^{q'}_{\omega^{q'/2}}(\R^n)} \le C_1 \norm{((-\Delta_\omega)^{\frac{\alpha}{2}} u)(t)}_{2,\omega}
    \quad \text{for all } t \in \R.
\end{equation}
Since $q' \le 2^\star$, and distinguishing between the cases $n \in \{1,2\}$ and $n \ge 3$, it is easy to see that $\alpha \in (0,1]$. Set $r = \frac{2}{\alpha} \in [2,\infty)$. This is equivalent to
\begin{equation}\label{eq: compatibilité}
\frac{1}{r} + \frac{n}{2q'} = \frac{n}{4}.
\end{equation}
By taking the $\L^r$-norm of \eqref{eq: HLS} and using \cite[Proposition 5.4]{auscherbaadi2024fundamental} with $S = (-\Delta_\omega)^{1/2}$ therein, there exists a constant $C_2 = C_2(r)$ such that
\begin{align*}
\| u \|_{\L^r(\R;\L^{q'}_{\omega^{q'/2}}(\R^n))}
\leq C_1 \left\| (-\Delta_\omega)^{\frac{\alpha}{2}} u \right\|_{\L^r(\R;\L^2_\omega(\R^n))}
\leq C_1 C_2 \|\nabla_x u \|_{2,\mu}^\alpha \|D_t^{1/2}u \|_{2,\mu}^{1-\alpha}  \le C_1 C_2 \norm{\D u}_{2,\mu},
\end{align*}
where we used the trivial identity $\norm{\nabla_x \cdot}_{2,\omega} = \norm{(-\Delta_\omega)^{\frac{1}{2}} \cdot}_{2,\omega}$ in the penultimate inequality. The case $q=2_\star$ follows directly, since solving for $r=q'$ in \eqref{eq: compatibilité} yields $r=q'=2^\star$. The case $q\in(2_\star,2)$ follows by interpolation. First, taking the $\L^2$-norm of \eqref{eq: HLS} and using an interpolation inequality (the moment inequality) yields
\begin{align*}
    \| u \|_{\L^2(\R;\L^{q'}_{\omega^{q'/2}}(\R^n))}
\leq C_1 \| (-\Delta_\omega)^{\frac{\alpha}{2}} u \|_{\L^2(\R;\L^2_\omega(\R^n))} &\le C_1 \left\|  u \right\|^{1-\alpha}_{\L^2(\R;\L^2_\omega(\R^n))} \| (-\Delta_\omega)^{\frac{1}{2}} u \|^\alpha_{\L^2(\R;\L^2_\omega(\R^n))} \\&= C_1 \left\|  u \right\|^{1-\alpha}_{2,\mu} \left\| \nabla_x u \right\|^\alpha_{2,\mu}
\\& \le C_1 \left\|  u \right\|^{1-\alpha}_{2,\mu} \left\| \D u \right\|^\alpha_{2,\mu},
\end{align*}
To conclude, note that since $q' \le 2^\star$, we have $q' \in (2,r]$, and we write $\frac{1}{q'} = \frac{1-\theta}{2} + \frac{\theta}{r}$ with $\theta \in (0,1]$. Therefore, by H\"older's inequality and the previous inequalities, 
\begin{align*}
    \norm{u}_{\L^{q'}_{\mu_{q'}}(\R^{n+1})} \le \norm{u}^\theta_{\L^r(\R;\L^{q'}_{\omega^{q'/2}}(\R^n))}  \| u \|^{1-\theta}_{\L^2(\R;\L^{q'}_{\omega^{q'/2}}(\R^n))} \le C \norm{\D u}^{\theta+\alpha(1-\theta)}_{2,\mu} \norm{u}^{(1-\alpha)(1-\theta)}_{2,\mu},
\end{align*}
where $C = C([\omega]_{A_2}, [\omega]_{RH_{\frac{q'}{2}}}, q, n)$ is a constant. As $\frac{1}{q'} = \frac{1-\theta}{2} + \frac{\theta}{2}\alpha$, we have $\theta(1-\alpha)= 1-\frac{2}{q'}$. The result follows since
\begin{equation*}
    \theta+\alpha(1-\theta)=\alpha+\theta(1-\alpha)
    = n\left(\frac{1}{2}-\frac{1}{q'}\right)
      +2\left(\frac{1}{2}-\frac{1}{q'}\right)=\kappa .
\end{equation*}
\end{proof}

\begin{rem}
The above result combines Hardy--Littlewood--Sobolev estimates in time and space. For the latter, if one directly uses existing results from the literature, namely \cite{muckenhoupt1974weighted} followed by interpolation \cite{auscher1997holomorphic} (see \cite[Corollary~4.3]{auscherbaadi2025hardy} for this argument), one is then led to assume that $n \ge 3$ and that $\omega$ belongs to a smaller reverse H\"older class. More precisely, setting $2^\star_e := \frac{2n}{n-2}$, the upper ``elliptic'' Sobolev 
conjugate of $2$, this argument requires
\begin{equation*}
    \omega \in RH_{\frac{n}{n-2}}(\R^{n}) 
    = RH_{\frac{2^\star_e}{2}}(\R^{n}) 
    \subsetneq RH_{\frac{2^\star}{2}}(\R^{n}) 
    \quad \text{since } \ 2^\star < 2^\star_e.
\end{equation*}
\end{rem}

We conclude by stating the following corollary, which is of independent interest.
\begin{cor}\label{cor: many functions}
If $\omega \in RH_{1+\frac{2}{n}}(\R^n)$, then $\L^{2_\star}_{\mu_{2_\star}}(\R^{n+1}) \cap \L^2_\mu(\R^{n+1}) \subset \operatorname{ran}(\sqrt{\mathcal{H}})$.
\end{cor}
\begin{proof}
We will prove this for $\H^*$ in place of $\H$, which belongs to the same class of operators. Since $\H$ is injective, we have $(\sqrt{\H^*})^{-1} = (\H^*)^{-\nicefrac 12} = (\H^{-\nicefrac 12})^*$ as closed operators in the sectorial functional calculi for $\H$ and $\H^*$, see e.g.\ \cite{egert2024harmonic}. Thus, the claim is that $\L^{2_*}_{\mu_{2_\star}}(\R^{n+1}) \cap \L^2_\mu(\R^{n+1}) \subset \mathrm{dom}((\H^{-\nicefrac 12})^*)$. 

\noindent Fix $u \in \L^{2_\star}_{\mu_{2_\star}}(\R^{n+1}) \cap \L^2_\mu(\R^{n+1})$. Proposition~\ref{lem: Sobolev} and the $\L^2_\mu$-boundedness of the Riesz transform yield constants $C_1=C_1([\omega]_{RH_{1+\frac{2}{n}}},[\omega]_{A_2},n)$ and $C_2=C_2(M,\nu,[\omega]_{A_2},n)$ such that
\begin{align*}
\abs{\langle u , \H^{-\nicefrac{1}{2}} v \rangle_{2,\mu}} \leq \norm{u}_{\L^{2_\star}_{\mu_{2_\star}}(\R^{n+1})} \norm{\H^{-\nicefrac{1}{2}} v}_{\L^{2^\star}_{\mu_{2^\star}}(\R^{n+1})} &\leq C_1 \norm{u}_{\L^{2_\star}_{\mu_{2_\star}}(\R^{n+1})} \norm{\D \H^{-\nicefrac{1}{2}} v}_{2,\mu} \\&\leq C_1 C_2 \norm{u}_{\L^{2_\star}_{\mu_{2_\star}}(\R^{n+1})} \norm{v}_{2,\mu},
\end{align*}
for all $v \in \mathrm{dom}(\H^{-\nicefrac{1}{2}})$ and the claim follows.     
\end{proof}
By the above result, if $\omega \in RH_{1+\frac{2}{n}}(\R^n)$, then $\mathcal{R}_\H$ can be defined on $\Cont_0^\infty(\R^{n+1})$ and extended by density, rather than being defined on an abstract dense subspace of $\L_\mu^2(\R^{n+1})$.

\subsection{Proof of \ref{Point (1)} of Theorem \ref{thm: théorème principal}}\label{s42}
Fix $q \in [2_\star,2)$ and assume that $\omega \in RH_{\frac{q'}{2}}(\R^n)$. By \eqref{eq: Gagliardo-Nirenberg} in Proposition \ref{lem: Sobolev}, there exists a constant $C = C([\omega]_{A_2}, [\omega]_{RH_{\frac{q'}{2}}}, q,n)$ such that
\begin{align*}
    \norm{\omega^{\frac{1}{2}} \mathcal{E}_\lambda \, \omega^{-\frac{1}{2}} u}_{q'}=\norm{\mathcal{E}_\lambda \, \omega^{-\frac{1}{2}} u}_{\L_{\mu_{q'}}^{q'}(\R^{n+1})} &\le C \norm{\mathbb{D} \E_\lambda \, \omega^{-\frac{1}{2}} u}^{-\gamma_{2,q'}}_{2,\mu} \norm{ \E_\lambda \, \omega^{-\frac{1}{2}} u}^{1+\gamma_{2,q'}}_{2,\mu} 
    \\&= C \lambda^{\gamma_{2,q'}}  \norm{\lambda\mathbb{D} \E_\lambda \, \omega^{-\frac{1}{2}} u}^{-\gamma_{2,q'}}_{2,\mu} \norm{ \E_\lambda \, \omega^{-\frac{1}{2}} u}^{1+\gamma_{2,q'}}_{2,\mu} 
    \\&\le \widetilde{C} \, \lambda^{\gamma_{2,q'}} \norm{u}_{2},
\end{align*}
for all $\lambda>0$ and $u\in \L^2(\R^{n+1})$, where we used Proposition \ref{prop: classical ODEs} in the last inequality, and $\widetilde{C}=\widetilde{C}(M,\nu,[\omega]_{A_2}, [\omega]_{RH_{\frac{q'}{2}}}, n)$ a constant. Thus, the family $\left ( \omega^{\frac{1}{2}} \mathcal{E}_\lambda \, \omega^{-\frac{1}{2}} \right )_{\lambda>0}$ is $\L^{2}-\L^{q'}$ bounded. The same holds for the adjoint family $\left( \omega^{\frac{1}{2}} \mathcal{E}^\star_\lambda \, \omega^{-\frac{1}{2}} \right)_{\lambda>0}$, and therefore $\left ( \omega^{\frac{1}{2}} \mathcal{E}_\lambda \, \omega^{-\frac{1}{2}} \right )_{\lambda>0}$ is $\L^{q}-\L^2$ bounded by adjointness. We recall that the adjointess of $\left (  \mathcal{E}_\lambda\right )_{\lambda>0}$ and $\left(\mathcal{E}^\star_\lambda \right)_{\lambda>0}$ is with respect to $\langle \cdot, \cdot \rangle_{2,\mu}$.

By interpolation with the $\L^2$ off-diagonal estimates in Proposition~\ref{prop: classical ODEs}, the family $\bigl( \omega^{\frac{1}{2}} \mathcal{E}_\lambda \, \omega^{-\frac{1}{2}} \bigr)_{\lambda>0}$ satisfies $\L^r-\L^2$ off-diagonal estimates for all $r \in (q,2]$. In particular, by Point~(4) of Lemma~\ref{lem: Toolbox}, this family is $\L^r$-bounded for all $r \in (q,2]$, and therefore $p_-(\H) \le q$. We are done if $q=2_\star$. If $q\in (2_\star,2)$, to prove that $p_-(\H) < q$, we use the open-ended property of reverse H\"older classes (Proposition~\ref{prop:weights}, Point~(2)) to pick $\tilde{q}\in (2_\star,q)$ such that $\omega \in RH_{\frac{\tilde{q}'}{2}}(\R^n)$. Then, by the same reasoning, $p_-(\H) \le \tilde{q}<q$. The other integrability condition in the definition of $p_-(\H)$ is omitted, as it is trivial.

Finally, to prove that $p_-(\H)\in [1,2)$, by Proposition~\ref{prop:weights}, Point~(3), we have $\omega\in RH_{\frac{\varrho_\omega'}{2}}(\R^n)$ for some $\varrho_\omega \in (2_\star,2)$ depending only on $[\omega]_{A_2}$ and $n$. Hence $p_-(\H) < \varrho_\omega < 2$.

\qed

\section{Riesz transform boundedness: Proof of \ref{Point (2)} of Theorem \ref{thm: théorème principal}}\label{s5}

Before sketching the proof of Point~\ref{Point (2)} of Theorem~\ref{thm: théorème principal}, we first recall two key ingredients: the validity of Theorem~\ref{thm: parabolicODEs} in the range of $p$ for which we establish the boundedness of the parabolic Riesz transform, and some preliminary facts concerning the functional calculus. All arguments follow either \cite{BEK26} or \cite{AEbook2023}, up to conjugation by $\omega^{\frac{1}{2}}$ and $\omega^{-\frac{1}{2}}$, and we refer to \cite[Sect. 6 and 7]{BEK26} for details.

Throughout this section, we fix $\varrho_\omega \in (1,2)$ such that the family $\left( \omega^{\frac{1}{2}} \mathcal{E}_\lambda \, \omega^{-\frac{1}{2}} \right)_{\lambda>0}$ is $\L^{\varrho_\omega}-\L^2$ bounded. Such an exponent always exists and depends only on $[\omega]_{A_2}$ and $n$. See Section \ref{s42}.

\subsection{Two Key Tools}

\subsubsection{The Critical Numbers}

We define the following exponent:
\begin{align*}
    q_-(\H) 
    &:= \inf \left\{ p \in( 1,\infty) : \ \text{the family} \ (\omega^{\frac{1}{2}}\lambda \mathbb{D} \E_\lambda \, \omega^{-\frac{1}{2}} )_{\lambda>0} \ \text{is} \ \L^p  \text{-bounded} \ \text{and} \ \mu_{p'} \ \text{is l.f. on } \R^{n+1}  \right\}.
\end{align*}
Recall that $p_-(\H)$, defined in the introduction, is as above with $\E_\lambda$ replacing $\lambda \mathbb{D}\E_\lambda$. 
\begin{prop}\label{prop: p_=q_}
We have $q_-(\H) \le p_-(\H)$.
\end{prop}

Let us recall the required tools. The first is the following lemma based on a Calder\'on reproducing formula. Although we only need the case $p=q$, we state it in full generality.
\begin{lem}[{\cite[Lemma 6.5]{AEbook2023}}]\label{lem: sans m}
Let $1\le p \le q <\infty$ with $\frac{n+2}{p}-\frac{n+2}{q}<1$. If there exists an integer $m\ge1$ such that the family $(\omega^{\frac{1}{2}} \lambda \mathbb{D}\mathcal{E}_\lambda^{m+1}\omega^{-\frac{1}{2}})_{\lambda>0}$ is $\L^p-\L^q$ bounded, then the family $(\omega^{\frac{1}{2}} \lambda \mathbb{D}\mathcal{E}_\lambda \omega^{-\frac{1}{2}})_{\lambda>0}$ is also $\L^p-\L^q$ bounded.
\end{lem}

The second key ingredient is the following criterion for extrapolating boundedness of operators with limited space-time decay.

\begin{lem}[{\cite[Proposition 5.10]{BEK26}}]\label{prop: MtMx}
Let $1\le p <  s <\infty$. Let $(T_\lambda)_{\lambda>0}$ be a family of operators on  $\L^s(\R^{n+1})$. Assume that for all $N>1$, there exists a constant $C$ such that for all $\lambda>0$, $j\ge 1$, and $u\in \L^p(\R^{n+1})\cap \L^s(\R^{n+1})$, the following two estimates hold:
\begin{equation*}
    \| \one_{\Delta_\lambda} T_\lambda (\one_{C^N_j(\Delta_\lambda)} u)\|_{s} \le C N^{-2j} \| \one_{C^N_j(\Delta_\lambda)} u \|_s,
\end{equation*}
\begin{equation*}
     \| T_\lambda  u\|_s \le C \lambda^{ \gamma_{p,s} } \|u \|_p.
\end{equation*}
For $\theta \in [0,1]$, if $\theta > \frac{1}{q_\theta}$, where $\frac{1}{q_\theta} := \frac{1-\theta}{p} + \frac{\theta}{s}$, then $(T_\lambda)_{\lambda>0}$ is $\L^r$-bounded for all $r \in (q_\theta, s]$.
\end{lem}

Combining Lemma~\ref{prop: MtMx} with (ii) of Theorem~\ref{thm: parabolicODEs}, and then Lemma~\ref{lem: sans m}, we obtain Proposition~\ref{prop: p_=q_} via an iteration argument starting from $2$, as in Step 2 of the proof of~\cite[Theorem 6.1]{BEK26}. The only difference is that the operators are conjugated by $\omega^{\frac{1}{2}}$ and that $\L^{\varrho_\omega}-\L^2$ boundedness is used instead of $\L^{2\star}-\L^2$ boundedness therein. We omit the details.

\subsubsection{The functional calculus}

By Theorem~\ref{thm: théorie L2}, the operator $\H$ is maximal accretive and thus admits a holomorphic functional calculus (see \cite{McIntosh86, haase2006functional, egert2024harmonic}). 

\begin{lem}[Functional calculus]\label{lem: funct calculus}
Let $p\in[1,2]$ and assume that the family $(\omega^{\frac{1}{2}}\,\mathcal{E}_\lambda \, \omega^{-\frac{1}{2}})_{\lambda>0}$ is $\L^p$-bounded.
Fix $\alpha >0$ and $\beta \ge 0$ and define the following holomorphic functions on $\mathbb{C}\setminus\{-1\}$ by:
\begin{align*}
\psi(z):= z^{3\alpha}(1+z)^{-6\alpha}, \quad \varphi(z):= (1-(1+z)^{-\beta})^{3\alpha}.
\end{align*}
Then, for all $q\in(p,2]$, all $\lambda,r>0$, all $u\in \L^q(\R^{n+1}) \cap \L^2(\R^{n+1})$, and all measurable sets $E,F\subseteq\R^{n+1}$, the following estimates hold with a constant $C=C(\alpha,\beta,p,q)$:
\begin{enumerate}[(i)]
\item $\norm{\one_F \omega^{\frac{1}{2}} \, \psi(\lambda^2 \H) \varphi(r^2\H)  \left( \omega^{-\frac{1}{2}} \one_E u \right)}_q \lesssim_q C \Bigl(1 + \frac{\d(E,F)}{\min(\lambda,r)}\Bigr)^{-6\alpha} \norm{\one_E u}_q$,
\item $\norm{\omega^{\frac{1}{2}} \, \psi(\lambda^2 \H) \varphi(r^2\H)  \left( \omega^{-\frac{1}{2}} u \right)}_q \lesssim_q C \min\left(1, \, \left(\frac{r}{\lambda}\right)^{2\alpha} \right) \norm{u}_q$.
\end{enumerate} 
\end{lem}

\subsection{Proof of Point \ref{Point (2)} of Theorem~\ref{thm: théorème principal}}
The extrapolation tool here is the following two-scale Blunck--Kunstmann theorem. In the homogeneous-scale case, this is due to \cite{blunck2003calderon} (see also \cite{auscher2007necessary}). The present two-scale version is a special case of \cite[Theorem 5.1]{BEK26}.

\begin{thm}[Two--scale Blunck--Kunstmann extrapolation]\label{thm: BK}
Let $1 \le p < q < \infty$. Let $T : \L^q(\R^{n+1})^{m_1} \to \L^q(\R^{n+1})^{m_2}$ be a bounded sublinear operator, where $m_1, m_2 \in \mathbb{N}^\star$. Let $(A_r)_{r>0}$ be a family of bounded linear operators on $\L^q(\R^{n+1})^{m_1}$. Let $N>1$. Assume that there exists a sequence $(g(j))_{j \ge 1} \in (0,\infty)^{\mathbb{N}^\star}$ such that for all $j \ge 2$
\begin{align}\label{eq: BK1}
\norm{\one_{C^N_j(\Delta_r)} T(1-A_{r}) (\one_{\Delta_r} u)}_q \le g(j) r^{\gamma_{p,q}} \|\one_{\Delta_r} u \|_p	
\end{align}
and for all $j \ge 1$
\begin{align}\label{eq: BK2}
\norm{\one_{C^N_j(\Delta_r)}A_{r} (\one_{\Delta_r} u)}_q \le g(j) r^{\gamma_{p,q}} \|\one_{\Delta_r} u \|_p	
\end{align}
for all parabolic cube $\Delta_r$ with radius $r>0$ and all $u\in \L^q(\R^{n+1})^{m_1}$. If
\begin{equation*}
    \Sigma:=\sum_{j=1}^{+\infty} g(j) \left ( 2^n N^2  \right )^{\frac{j}{q'}}<\infty,
\end{equation*}
then $T$ is of weak-type $(p,p)$, with a bound depending only on $\norm{T}_{\mathcal{L}(\L^q)}$, $\Sigma$, $N$, $n$ and $q$. In particular, $T: \L^r(\R^{n+1})^{m_1} \rightarrow  \L^r(\R^{n+1})^{m_2}$ is  of strong type $(r,r)$, for all $r\in (p,q]$.
\end{thm}

\begin{proof}[Proof of Point \ref{Point (2)} of Theorem~\ref{thm: théorème principal}]
With this criterion, (i) and (i) bis of Theorem~\ref{thm: parabolicODEs} and Lemma~\ref{lem: funct calculus}, the proof of Point \ref{Point (2)} of Theorem~\ref{thm: théorème principal} follows exactly as in \cite[Theorem 7.3]{BEK26}, by conjugation with $\omega^{\frac{1}{2}}$ and $\omega^{-\frac{1}{2}}$ and using $\L^{\varrho_\omega}-\L^2$ boundedness instead of $\L^{2^\star}-\L^2$ boundedness therein. More precisely, one proves that if $\omega^{\frac{1}{2}} \mathcal{R}_\H \omega^{-\frac{1}{2}}$ is $\L^q$-bounded, then Theorem~\ref{thm: BK} applies with $T=\omega^{\frac{1}{2}} \mathcal{R}_\H \omega^{-\frac{1}{2}}$ by using the same Calder\'on reproducing formula as in \cite[Theorem 7.3]{BEK26} and the same operators $A_r:=\omega^{\frac{1}{2}} \,(1-\varphi(r^2\H)) \, \omega^{-\frac{1}{2}}$ as therein, conjugated by $\omega^{\frac{1}{2}}$ and $\omega^{-\frac{1}{2}}$. This yields the $\L^p$-boundedness of $\omega^{\frac{1}{2}} \mathcal{R}_\H \omega^{-\frac{1}{2}}$ for all $p \in (q_\star, q] \cap (p_-(\mathcal{H}), 2]$, where $q_\star$ is the lower parabolic Sobolev conjugate of $q$, defined by
\begin{equation*}
    \frac{1}{q_\star} - \frac{1}{q} = \frac{1}{n+2}.
\end{equation*}
We then iterate the argument, starting from $q=2$ since $\omega^{\frac{1}{2}} \mathcal{R}_\H \omega^{-\frac{1}{2}}$ is $\L^2$-bounded, and deduce that for all $p \in (p_-(\mathcal{H}),2]$, it is $\L^p$-bounded, which is equivalent to the fact that $\mathcal{R}_\mathcal{H}$ extends to a bounded operator on $\L^p_{\mu_p}(\R^{n+1})$. We skip the details.
\end{proof}

\section{Real-valued coefficients: Proof of \ref{Point (3)} of Theorem \ref{thm: théorème principal}}\label{s6}

In this section, we assume that the matrix-valued function $A$ has real-valued coefficients. The key point here is the availability of Gaussian bounds, which we state next. For $x\in\R^n$ and $r>0$, we set
\begin{equation*}
    \omega_r(x):=\omega(B(x,\sqrt{r})):=\int_{B(x,\sqrt{r})}\omega(y)  \dd y,
\end{equation*}
where $B(x,r)$ is the Euclidean ball of radius $r$ and center $x$.

\begin{lem}\label{lem: bornes gaussiennes} 
Fix $\lambda\in\C$ with $\mathrm{Re}(\lambda^{-2})>0$ and $m\geq1$. Then the resolvent $\mathcal{E}_\lambda^m$ is an integral operator on $\L^2_\mu(\R^{n+1})$ with kernel $K_\lambda^m$ satisfying pointwise estimates in the following cases.
\begin{enumerate}[label=(\arabic*), leftmargin=*, itemsep=0pt]
    \item \textbf{Upper and lower Gaussian bounds on $\R_+\setminus\{0\}$:} If $\lambda>0$, then
    \begin{equation*}
    \frac{1}{C}\mathbf{1}_{t>s} \frac{e^{-\frac{t-s}{\lambda^2}}}{\lambda^{2m}}(t-s)^{(m-1)}\frac{e^{-c\frac{|x-y|^2}{t-s}}}{\omega_{t-s}(y)}\le K^m_\lambda(x,t;y,s)\le C \, \mathbf{1}_{t>s} \frac{e^{-\frac{t-s}{\lambda^2}}}{\lambda^{2m}}(t-s)^{(m-1)}\frac{e^{-\frac{1}{c}\frac{|x-y|^2}{t-s}}}{\omega_{t-s}(y)},
    \end{equation*}
    where $C>1$ and $c>1$ are constants depending only on $M$, $\nu$, $[\omega]_{A_2}$, $n$, and $m$.
    \item \textbf{Upper Gaussian bounds on sectors:} Fix $\mu\in(0,\nicefrac{\pi}{4})$ and set $S^+_{\mu}:=\{z\in\C\setminus\{0\}:|\arg(z)|<\mu\}$. If $\lambda\in S^+_{\mu}$, then
    \begin{equation*}
    |K^m_\lambda(x,t;y,s)|\le C \, \mathbf{1}_{t>s} \frac{e^{-c_\mu\frac{t-s}{|\lambda|^2}}}{|\lambda|^{2m}}(t-s)^{(m-1)}\frac{e^{-\frac{1}{c}\frac{|x-y|^2}{t-s}}}{\omega_{t-s}(y)},
    \end{equation*}
    where $C$ and $c$ are as in (1), and $c_\mu\in(0,1]$ depends only on $\mu$.
\end{enumerate} 
\end{lem}
The case $\lambda>0$ in (1) above is the one most used in what follows. The case $\lambda\in S^+_{\mu}$ is used for the weak-type $(1,1)$ estimate. The proof is given in Appendix \ref{annexe 1}.
\begin{rems}\label{rem: Cruz}
\begin{enumerate}[label=(\arabic*), leftmargin=*, itemsep=0pt]
    \item  The factor $\frac{1}{\omega_{t-s}(y)}$ in the upper bounds in (1) and (2) of Lemma
    \ref{lem: bornes gaussiennes} above may be replaced by one of the following factors
    \begin{equation*}
    \frac{1}{\omega_{t-s}(x)}, \ \ \frac{1}{\sqrt{\omega_{t-s}(x)}\sqrt{\omega_{t-s}(y)}}, \ \ \frac{1}{\max(\omega_{t-s}(x),\omega_{t-s}(y))},
    \end{equation*}
    and $C$ and $c$ are replaced respectively by $\widetilde{C}=\widetilde{C}(C,c,[\omega]_{A_2})$ and $2c$. For details, see \cite[Rem. 3]{cruz2014corrigendum}. The same statement holds for the lower bound in (1), with a minimum in the third option. We will use this fact freely.
    \item By adjointness, Lemma \ref{lem: bornes gaussiennes} also holds for $(\mathcal{E}^\star_\lambda)^m$ by interchanging $t$ and $s$, and $x$ and $y$.
\end{enumerate}
\end{rems}
We also state the following general lemma, which is valid for complex coefficients. Its proof, and that of Lemma~\ref{lem: sans m}, is based on the same Calder\'on reproducing formula and follows \emph{verbatim} from \cite[Lemma 6.5]{AEbook2023}; although we only need the case $p=q$, we state it in full generality.
\begin{lem}\label{lem2: sans m}
Let $1\le p \le q <\infty $ with $\frac{n+2}{p}-\frac{n+2}{q}<1$. Assume that there exists an integer $m\ge1$ such that the family $(\omega^{\frac{1}{2}} \mathcal{E}_\lambda^{m+1} \omega^{-\frac{1}{2}})_{\lambda>0}$ is $\L^p-\L^q$ bounded. Then, the family $( \omega^{\frac{1}{2}} \mathcal{E}_\lambda  \omega^{-\frac{1}{2}})_{\lambda>0}$ is $\L^p-\L^q$ bounded.
\end{lem}

Finally, we will purposedly omit to verify the local finiteness of $\mu_{p'}$ in the definition of $p_-(\mathcal{H})$, as this will be trivially guaranteed by the reverse H\"older condition.

\begin{proof}[Proof of Point~\ref{Point (3)} of Theorem~\ref{thm: théorème principal}]
We divide the proof into several steps.
\newline
\paragraph{\textit{\textbf{Step 1:} For all $q \in (1,2]$, $\omega \in RH_{\frac{q'}{2}}(\R^n) \, \Rightarrow \, p_-(\H) < q$.}} We fix $q\in (1,2]$ and assume that $\omega \in RH_{\frac{q'}{2}}(\R^n)$. For $m \ge 1$, $u \in \L^2_\mu(\R^n)$, we set, for all $(x,t)\in \R^{n+1}$,
\begin{equation*}
    v(x,t):= \lambda^{-2m} \iint_{\R^{n+1}} \one_{t>s}\, e^{-\frac{t-s}{\lambda^2}} (t-s)^{(m-1)}\frac{e^{-\frac{1}{c}\frac{|x-y|^2}{t-s}}}{\omega_{t-s}(y)} |u(s,y)| \d \omega(y) \d s.
\end{equation*}
Using Fubini's theorem and then Minkowski's inequality, we have for all $t\in \R$,
\begin{align*}
    \norm{v(\cdot,t)}_{\L^{q'}_{\omega^{q'/2}}(\R^n)} &\le\lambda^{-2m} \int_{\R} \one_{t>s}\, e^{-\frac{t-s}{\lambda^2}} (t-s)^{(m-1)}  \left\| \int_{\R^n}\frac{e^{-\frac{1}{c}\frac{|\cdot-y|^2}{t-s}}}{\omega_{t-s}(y)} |u(s,y)| \d \omega(y) \right\|_{\L^{q'}_{\omega^{q'/2}}(\R^n)} \!\!\!\!\!\!\!\! \d s
    \\& \le C \lambda^{-2m} \int_{\R} \one_{t>s}\, e^{-\frac{t-s}{\lambda^2}} (t-s)^{(m-1)} (t-s)^{\frac{n}{2}(\frac{1}{q'}-\frac{1}{2})}\norm{u(\cdot,s)}_{2,\omega} \d s,
\end{align*}
where $C = C([ \omega ]_{A_2},[ \omega ]_{RH_{\frac{q'}{2}}},n,q)$, and we have used \cite[Lemma 3.7]{auscherbaadi2025hardy} in the last line, since $\omega \in RH_{\frac{q'}{2}}(\R^n)$. Using Young's convolution inequality on $\R$ with $\ell>1$ such that $\frac{1}{2}+\frac{1}{\ell}=1+\frac{1}{q'}$, we deduce that 
\begin{align*}
    \norm{v}_{\L^{q'}_{\mu_{q'}}(\R^{n+1})} &\le C \lambda^{-2m} \left ( \int_{0}^{+\infty} e^{-\ell\frac{t}{\lambda^2}} \, t^{\ell(m-1)} \, t^{\frac{\ell n}{2}(\frac{1}{q'}-\frac{1}{2})} \d t  \right )^{1/\ell} \norm{u}_{2,\mu} \\&= C \left ( \int_{0}^{+\infty} e^{-\ell t} \, t^{\ell(m-1)} \, t^{\frac{\ell n}{2}(\frac{1}{q'}-\frac{1}{2})} \d t  \right )^{1/\ell} \lambda^{(n+2)\left( \frac{1}{q'}-\frac{1}{2} \right)} \, \norm{u}_{2,\mu}.
\end{align*}
We then choose $m\ge 1$ so that the above integral is finite (integrability near $0$). By the upper bound in (1) of Lemma~\ref{lem: bornes gaussiennes} and the above computations, we now know that the family $\bigl( \omega^{\frac{1}{2}} \E^m_\lambda \, \omega^{-\frac{1}{2}} \bigr)_{\lambda>0}$ is $\L^2-\L^{q'}$ bounded. The same is true for the family $\bigl( \omega^{\frac{1}{2}} (\E^\star_\lambda)^m \, \omega^{-\frac{1}{2}} \bigr)_{\lambda>0}$. By adjointess, the family $\bigl( \omega^{\frac{1}{2}} \E^m_\lambda \, \omega^{-\frac{1}{2}} \bigr)_{\lambda>0}$ is then $\L^q-\L^2$ bounded. Fix $p\in (q,2]$. By interpolating with the $\L^2$ off-diagonal estimates for the family $(\omega^{\frac{1}{2}} \,\E^{m}_\lambda \, \omega^{-\frac{1}{2}})_{\lambda>0}$ (see Proposition~\ref{prop: classical ODEs} and Lemma~\ref{lem: Toolbox}, Point~(2)), we deduce that this family satisfies $\L^p-\L^2$ off-diagonal estimates. Consequently, it is $\L^p$-bounded by Lemma~\ref{lem: Toolbox}, Point~(3). By Lemma~\ref{lem2: sans m}, the family $(\omega^{\frac{1}{2}} \,\E_\lambda \, \omega^{-\frac{1}{2}})_{\lambda>0}$ is $\L^p$-bounded. As this true for all $p\in (q,2]$, then $p_-(\H) \le q$. Finally, to prove that $p_-(\H) < q$, we use the open-ended property of reverse H\"older classes (Proposition~\ref{prop:weights}, Point~(2)) to pick $\tilde{q}\in (1,q)$ such that $\omega \in RH_{\frac{\tilde{q}'}{2}}(\R^n)$. Then, by the same reasoning, $p_-(\H) \le \tilde{q}<q$.
\newline
\paragraph{\textit{\textbf{Step 2:} For all $q \in (1,2]$, $ p_-(\H) < q \, \Rightarrow \, \omega \in RH_{\frac{q'}{2}}(\R^n)$.}} 
We fix $q \in (1,2]$ and assume that $p_-(\H) < q$. Let $\tilde{q} \in (p_-(\H),q)$. By Section~\ref{s42}, there exists $\varrho_\omega=\varrho_\omega([\omega]_{A_2},n) \in (1,2)$ such that the family $\bigl(\omega^{\frac{1}{2}} \mathcal{E}_\lambda \,\omega^{-\frac{1}{2}}\bigr)_{\lambda>0}$ is $\L^{\varrho_\omega}-\L^2$ bounded. Since it is also $\L^{\tilde{q}}$- and $\L^2$-bounded, Lemma~\ref{lem: gives m} yields the existence of $m \ge 1$ such that the family $\bigl( \omega^{\frac{1}{2}} \E^m_\lambda \, \omega^{-\frac{1}{2}} \bigr)_{\lambda>0}$ is $\L^q-\L^2$ bounded. By adjointess, $\bigl( \omega^{\frac{1}{2}} (\E^\star_\lambda)^m \, \omega^{-\frac{1}{2}} \bigr)_{\lambda>0}$ is $\L^2-\L^{q'}$ bounded. Since the conclusion will be the same, we assume for simplicity that $\bigl(\omega^{\frac{1}{2}} \mathcal{E}_\lambda^m \,\omega^{-\frac{1}{2}}\bigr)_{\lambda>0}$ is $\L^2-\L^{q'}$ bounded. Using the lower bound in (1) of Lemma~\ref{lem: bornes gaussiennes}, we have
\begin{align*}
    \lambda^{-2q'm} \iint_{\R^{n+1}} &\left ( \iint_{\R^{n+1}} \mathbf{1}_{t>s}
e^{-\frac{t-s}{\lambda^2}}
(t-s)^{(m-1)}\frac{e^{-c\frac{|x-y|^2}{t-s}}}{\omega_{t-s}(y)}  u(y,s) \d \omega(y) \d s \right)^{q'} \omega^{\frac{q'}{2}}(x) \d x \d t \\&\hspace{9cm}\le C^{q'} \lambda^{(n+2)\left( 1-\frac{q'}{2} \right)} \norm{u}^{q'}_{2,\mu},
\end{align*}
for all nonnegative $u \in \L^2_\mu(\R^{n+1}) \cap \L^{q'}_{\mu_{q'}}(\R^{n+1})$. Let $Q \subset \R^n$ be a cube with sides parallel to the coordinate axes and side length $\lambda >0$. Taking $u=\one_{Q\times [0,\lambda^2]}$ yields
\begin{align*}
&\lambda ^{-2q'm}  \int_{2\lambda^2}^{3\lambda^2}  \int_Q  \left ( \int_{0}^{\lambda^2}  \int_Q e^{-\frac{t-s}{\lambda^2}}(t-s)^{(m-1)}\frac{e^{-c\frac{|x-y|^2}{t-s}}}{\omega_{t-s}(y)} \d \omega(y) \d s  \right)^{q'}  \omega^{\frac{q'}{2}}(x) \d x \d t \\&\hspace{9cm}\le C^{q'} \lambda^{(n+2)\left( 1-\frac{q'}{2} \right)}(\lambda^2 \omega(Q))^{\frac{q'}{2}}.
\end{align*}
Since $\lambda^2 \le t-s \le 3\lambda^2$, we obtain
\begin{align*}
  \int_Q \left (  \int_Q \frac{e^{-c\frac{|x-y|^2}{\lambda^2}}}{\omega_{3\lambda^2}(y)}  \d \omega(y)\right)^{q'} \omega^{\frac{q'}{2}}(x) \d x\le e^{3q'} C^{q'} \lambda^{n\left( 1-\frac{q'}{2} \right)} \omega(Q)^{\frac{q'}{2}}.
\end{align*}
Using the doubling and reverse doubling properties of $\omega$, together with norm equivalence on $\R^n$, we easily see that there exists a constant $\gamma=\gamma(n,[\omega]_{A_2},c)>0$ such that, for all $\lambda>0$,
\begin{equation*}
    \gamma \leq \inf_{x\in Q} \int_Q \frac{e^{-\frac{c|x-y|^2}{\lambda^2}}}{\omega_{3\lambda^2}(y)}  \d \omega(y).
\end{equation*}
Thus,
\begin{equation*}
    \int_{Q} \omega^{q'/2}(x) \, \mathrm d x \leq e^{3q'} \gamma^{-q'} C^{q'} |Q|^{1-\frac{q}{p}} \omega(Q)^{\frac{q'}{2}}. 
\end{equation*}
Therefore, $\omega \in RH_{\frac{q'}{2}}(\R^n)$. 
\newline
\paragraph{\textit{\textbf{Step 3:} $\omega \in RH_{\star}(\R^n)$ if and only if $p_-(\H)=1$.}} 
If $\omega \in RH_{\star}(\R^n)$, then by \textbf{Step~1} we have that $p_-(\H)<q$ for all $q\in(1,2]$. Therefore, $p_-(\H)=1$. Conversely, if $p_-(\H)=1$, then $\omega \in RH_{\star}(\R^n)$ by \textbf{Step~3}.
\newline
\paragraph{\textit{\textbf{Step 4:} The case $\omega \in RH_{\infty}(\R^n)$.}} We know that $RH_{\infty}(\R^n) \subset RH_{\star}(\R^n)$, but let us prove the desired results in a direct manner. Using the upper bounds in (1) of Lemma \ref{lem: bornes gaussiennes} with $m=1$, for every $\lambda>0$ and every $u \in \L^1(\R^{n+1}) \cap \L^2(\R^{n+1})$, Fubini's theorem yields
\begin{align*}
    &\norm{\omega^{\frac{1}{2}}\E_\lambda \, \omega^{-\frac{1}{2}} u}_1 \leq \iint_{\R^{n+1}} \left( \iint_{R^{n+1}} |K^1_\lambda(x,t;y,s)| |u(y,s)| \omega^{\frac{1}{2}}(y) \, \mathrm{d}y \, \mathrm{d}s \right) \omega^{\frac{1}{2}}(x) \, \mathrm{d}x \, \mathrm{d}t
    \\& \le C \iint_{\R^{n+1}} \left( \iint_{R^{n+1}} \one_{t>s}\, \frac{e^{-\frac{t-s}{\lambda^2}}}{\lambda^2} \frac{e^{-\frac{1}{c}\frac{|x-y|^2}{t-s}}}{\sqrt{\omega_{t-s}(x)}\sqrt{\omega_{t-s}(y)}} |u(y,s)| \omega^{\frac{1}{2}}(y) \, \mathrm{d}y \, \mathrm{d}s \right) \omega^{\frac{1}{2}}(x) \, \mathrm{d}x \, \mathrm{d}t \\
    &=C \iint_{\R^{n+1}} |u(y,s)|  \left( \iint_{R^{n+1}} \one_{t>s} \, \frac{e^{-\frac{t-s}{\lambda^2}}}{\lambda^2} e^{-\frac{1}{c} \frac{|x-y|^2}{t-s}}  \left ( \frac{\omega(x)}{\omega_{t-s}(x)} \right )^{\frac{1}{2}}  \left ( \frac{\omega(y)}{\omega_{t-s}(y)} \right )^{\frac{1}{2}} \mathrm{d}x \, \mathrm{d}t \right) \, \mathrm{d}y \, \mathrm{d}s.
\end{align*}
Now, since $\omega \in RH_{\infty}(\R^n)$, for all $s<t$ and for a.e.\ $x \in \R^n$ and for a.e.\ $y \in \R^n$, we have
\begin{equation}\label{eq: RHinfini}
    \left ( \frac{\omega(x)}{\omega_{t-s}(x)} \right )^{\frac{1}{2}}  \left ( \frac{\omega(y)}{\omega_{t-s}(y)} \right )^{\frac{1}{2}} \le \frac{[\omega]_{RH_\infty}}{(t-s)^{n/2}}. 
\end{equation}
By using this and a change of variables in the above computations, we obtain
\begin{equation*}
    \norm{\omega^{\frac{1}{2}}\E_\lambda \, \omega^{-\frac{1}{2}} u}_1 \leq C \, [\omega]_{RH_\infty} \left( {\pi}{c} \right)^{n/2} \norm{u}_1. 
\end{equation*}
It follows that $p_-(\H) = 1$. More than that, we may reuse the upper bounds in (1) of Lemma \ref{lem: bornes gaussiennes} to show that there exists $m \ge 1$ such that the family $(\omega^{\frac{1}{2}}\E^m_\lambda \, \omega^{-\frac{1}{2}})_{\lambda>0}$ satisfies $\L^1-\L^2$ off-diagonal estimates. Indeed, if $E,F \subset \R^{n+1}$ are measurable sets, $\d := \d(E,F)$ is the parabolic distance between $E$ and $F$, and $u \in \L^1(\R^{n+1}) \cap \L^2(\R^{n+1})$, then for all $m \ge 1$:
\begin{align*}
    &\norm{\one_F \, \omega^{\frac{1}{2}} \E_\lambda^m \, \omega^{-\frac{1}{2}} (\one_E u)}^2_2 
    \le \iint_{\R^{n+1}} \one_F \left( \iint_{R^{n+1}}|K^m_\lambda(x,t;y,s)| |\one_E u|(y,s) \,  \omega^{\frac{1}{2}}(y) \, \mathrm{d}y \, \mathrm{d}s \right)^2 \mathrm{d}\omega(x) \, \mathrm{d}t
    \\& \le C^2 \iint_{\R^{n+1}}\left( \iint_{R^{n+1}} \mathbf{1}_{t>s} \, \one_{|x-y|^2+|t-s| - \d^2 \ge 0} \, \frac{e^{-\frac{t-s}{\lambda^2}}}{\lambda^{2m}}
(t-s)^{(m-1)} e^{-\frac{1}{c}\frac{|x-y|^2}{t-s}} |\one_E u|(y,s)
\right. \\
    &\hspace{7cm} \left. \left ( \frac{\omega(x)}{\omega_{t-s}(x)} \right )^{\frac{1}{2}}  \left ( \frac{\omega(y)}{\omega_{t-s}(y)} \right )^{\frac{1}{2}}   \, \mathrm{d}y \, \mathrm{d}s  \right)^2  \, \mathrm{d}x \, \mathrm{d}t.
\end{align*}
Using \eqref{eq: RHinfini} and Young's convolution inequality on $\R^{n+1}$ yields:
\begin{align*}
   \norm{\one_F \, \omega^{\frac{1}{2}} \E_\lambda^m \, \omega^{-\frac{1}{2}} (\one_E u)}_2 &\le C \, [\omega]_{RH_\infty} \left( \iint_{R^{n+1}} \one_{t>0} \, \frac{e^{-2\frac{t}{\lambda^2}}}{\lambda^{4m}} \frac{e^{-\frac{2}{c} \frac{|x|^2}{t}}}{t^{n-2(m-1)}} \one_{|x|^2+t - \d^2\ge 0} \, \mathrm{d}x \, \mathrm{d}t \right)^{1/2} \norm{\one_E u}_1
    \\& =  \frac{C \, [\omega]_{RH_\infty}}{\lambda^{2m}} \left( \int^{\infty}_{0}  \frac{e^{-2\frac{t}{\lambda^2}}}{t^{n/2-2(m-1)}}  \left ( \int_{|x|^2\ge \frac{\d^2}{t}-1} e^{-\frac{2}{c} |x|^2} \, \mathrm{d}x \right ) \, \mathrm{d}t \right)^{1/2} \norm{\one_E u}_1
    \\& \leq e^{\frac{1}{2c}} \left( {\pi}{c} \right)^{n/4} \frac{C \, [\omega]_{RH_\infty}}{\lambda^{2m}} \left( \int^{\infty}_{0}  \frac{e^{-2\frac{t}{\lambda^2}} e^{-\frac{1}{c} \frac{\d^2}{t}}}{t^{n/2-2(m-1)}}  \, \mathrm{d}t \right)^{1/2} \norm{\one_E u}_1
    \\&= C \, [\omega]_{RH_\infty} e^{\frac{1}{2c}} \left( {\pi}{c} \right)^{n/4} \lambda^{\frac{n+2}{2}-\frac{n+2}{1}} \left( \int^{\infty}_{0}  \frac{e^{-2t} e^{-\frac{1}{c} \frac{\d^2}{\lambda^2} \frac{1}{t}}}{t^{n/2-2(m-1)}}  \, \mathrm{d}t \right)^{1/2} \norm{\one_E u}_1,
\end{align*}
where we have used a change of variables in $x$ in the second line and in $t$ in the last line, and wrote $e^{-\frac{2}{c}|x|^2}=e^{-\frac{1}{c}|x|^2}e^{-\frac{1}{c}|x|^2}$ in the second line to bound the expression by a Gaussian integral. We then choose $m \ge 1$ such that $n/2 - 2(m-1) < 1$, so that the above integral is finite (integrability near $0$ when $\d = 0$). The result then follows immediately by writing
\begin{align*}
    \int^{\infty}_{0}  \frac{e^{-2t} e^{-\frac{1}{c} \frac{\d^2}{\lambda^2} \frac{1}{t}}}{t^{n/2-2(m-1)}}  \, \mathrm{d}t &= \int^{\d /\lambda}_{0}  \frac{e^{-2t} e^{-\frac{1}{c} \frac{\d^2}{\lambda^2} \frac{1}{t}}}{t^{n/2-2(m-1)}}  \, \mathrm{d}t + \int^{\infty}_{\d /\lambda}  \frac{e^{-2t} e^{-\frac{1}{c} \frac{\d^2}{\lambda^2} \frac{1}{t}}}{t^{n/2-2(m-1)}}  \, \mathrm{d}t \\&\le  \left( \int^{\infty}_{0}  \frac{e^{-t} }{t^{n/2-2(m-1)}}  \, \mathrm{d}t \right) e^{-\frac{1}{c} \frac{\d}{\lambda}},
\end{align*}
where the last estimate follows from elementary inequalities.

Furthermore, using the same estimates as above, we show that the family $(\omega^{\frac{1}{2}}\mathcal{E}_\lambda \omega^{-\frac{1}{2}})_{\lambda>0}$ satisfies $\L^1$ off-diagonal estimates. By the same computations and using (2) of Lemma \ref{lem: bornes gaussiennes}, the family $(\omega^{\frac{1}{2}}\mathcal{E}_\lambda \omega^{-\frac{1}{2}})_{\lambda \in S^+_{\mu}}$, for any fixed $\mu \in (0,\nicefrac{\pi}{4})$, satisfies $\L^1$ off-diagonal estimates, with constants depending additionally on $\mu$. In particular, Lemma~\ref{lem: funct calculus} holds with $q=1$.

Finally, proving that the spatial gradient component $\omega^{\frac{1}{2}} \nabla_x \H^{-1/2}$ defines a bounded operator from $\L^1_{\mu_1}(\R^{n+1})$ to $\L^{1,\infty}(\R^{n+1})$ is equivalent to the weak-type $(1,1)$ boundedness of $\omega^{\frac{1}{2}}\nabla_x \mathcal{H}^{-1/2} \omega^{-\frac{1}{2}}$ from $\L^1(\R^{n+1})$ to $\L^{1,\infty}(\R^{n+1})$. To prove the latter, we use the above $\L^1-\L^2$ bounds for $(\omega^{\frac{1}{2}}\E^m_\lambda \, \omega^{-\frac{1}{2}})_{\lambda>0}$ and the $\L^1$ off-diagonal estimates of arbitrarily large order for the family $(\omega^{\frac{1}{2}}\mathcal{E}_\lambda \omega^{-\frac{1}{2}})_{\lambda \in S^+_{\mu}}$ for a fixed $\mu \in (0,\nicefrac{\pi}{4})$ (more precisely, Lemma~\ref{lem: funct calculus} with $q=1$ therein), together with the fact that the family $( \omega^{\frac{1}{2}} \lambda\nabla_x \mathcal{E}_\lambda \omega^{-\frac{1}{2}})_{\lambda>0}$ satisfies $\L^2$ off-diagonal estimates. These ingredients allow for the extrapolation of the weak-type $(1,1)$ boundedness of $\omega^{\frac{1}{2}}\nabla_x \mathcal{H}^{-1/2} \omega^{-\frac{1}{2}}$ from $\L^1(\R^{n+1})$ to $\L^{1,\infty}(\R^{n+1})$, exactly as in the classical elliptic case \cite[Sect.~7.1]{AEbook2023}. Indeed, one uses the extrapolation criterion of Theorem~\ref{thm: BK} in the classical homogeneous-scale case $N=2$, yielding the result directly, without resorting to any iterative argument. We omit the details.
\end{proof}

\begin{rem}
We have already seen that there exists $\varrho_\omega \in (1,2)$ such that the family $\left( \omega^{\frac{1}{2}} \mathcal{E}_\lambda \, \omega^{-\frac{1}{2}} \right)_{\lambda>0}$ is $\L^{\varrho_\omega}-\L^2$ bounded, but we can see this under Gaussian bounds directly. Indeed, by Proposition~\ref{prop:weights}, Point~(3), we have $\omega \in RH_{\frac{\varrho'}{2}}(\R^n)$ for some $\varrho \in (1,2)$ sufficiently close to $2$ so that
\begin{equation*}
    \frac{n}{2}\left(\frac{1}{2}-\frac{1}{\varrho'}\right) < 1+\frac{1}{\varrho'}-\frac{1}{2}.
\end{equation*}
Thus, for $\ell>1$ such that $\frac{1}{2}+\frac{1}{\ell}=1+\frac{1}{\varrho'}$, we have
\begin{equation*}
    \int_{0}^{+\infty} e^{-\ell t} \, t^{\frac{\ell n}{2}\left(\frac{1}{\varrho'}-\frac{1}{2}\right)} \d t < \infty.
\end{equation*}
We then proceed as in \textbf{Step~1} with $m=1$ to conclude that the family $\bigl( \omega^{\frac{1}{2}} \E_\lambda \, \omega^{-\frac{1}{2}} \bigr)_{\lambda>0}$ is $\L^\varrho-\L^2$ bounded.
\end{rem}

\begin{rem}
The Gaussian upper bounds on sectors in point (2) of Lemma \ref{lem: bornes gaussiennes} may not be used. They are only used in \textbf{Step~4} to obtain Lemma~\ref{lem: funct calculus} with $q=1$. However, $\L^1$-boundedness of $(\omega^{\frac{1}{2}}\mathcal{E}_\lambda \omega^{-\frac{1}{2}})_{\lambda>0}$ extends to a small sector $S^+_\mu$ by a Neumann series argument, see the proof of \cite[Lemma 7.1]{BEK26}. Interpolation with the $\L^1$ off-diagonal estimates on $(0,\infty)$ then gives $\L^1$ off-diagonal estimates on, say, $S^+_{\mu/2}$, which suffices to obtain Lemma~\ref{lem: funct calculus} with $q=1$. Nevertheless, since we had to prove the Gaussian bounds on $(0,\infty)$ anyway, and the proof directly gives the bounds on $S^+_\mu$ for any $\mu\in(0,\nicefrac{\pi}{4})$, we chose to state the Gaussian bounds as in Lemma~\ref{lem: bornes gaussiennes}, which makes the statement more explicit and the argument less involved.
\end{rem}

\section{Reverse inequalities via duality: Proof of Corollary~\ref{cor: reverse}}\label{s7}

Fix $u \in \Cont_0^\infty(\R^{n+1})$ and $p\in [2,p_-(\H^\star)')$. For $\phi \in \mathrm{ran}(\sqrt{\H^\star}) \cap \L^{p'}_{\mu_{p'}}(\R^{n+1})$, we have
\begin{align*}
    \langle \H^{\nicefrac{1}{2}} u , \phi \rangle_{2,\mu} &= \langle \H^{\nicefrac{1}{2}} u , (\H^\star)^{\nicefrac{1}{2}} (\H^\star)^{-\nicefrac{1}{2}} \phi \rangle_{2,\mu} \\&= \langle \H u ,  (\H^\star)^{-\nicefrac{1}{2}} \phi \rangle_{2,\mu} 
    \\& = \langle H_t D_t^{1/2} u , D_t^{1/2} (\H^\star)^{-\nicefrac{1}{2}} \phi \rangle_{2,\mu} + \langle \omega^{-1} A \nabla_x u , \nabla_x (\H^\star)^{-\nicefrac{1}{2}} \phi \rangle_{2,\mu}.
\end{align*}
Thus, since the Hilbert transform is bounded on $\L^r(\R)$ for every $1<r<\infty$, we obtain
\begin{equation*}
    |\langle \H^{\nicefrac{1}{2}} u , \phi \rangle_{2,\mu}| \lesssim \norm{\mathbb{D}u}_{\L^p_{\mu_p}(\R^{n+1})} \norm{\mathbb{D} (\H^\star)^{-\nicefrac{1}{2}} \phi}_{\L^{p'}_{\mu_{p'}}(\R^{n+1})}.
\end{equation*}
Now, since $p' \in (p_-(\H^\star),2]$, we obtain, by Point~\ref{Point (2)} of Theorem~\ref{thm: théorème principal} applied to $\H^\star$, that
\begin{equation*}
    |\langle \H^{\nicefrac{1}{2}} u , \phi \rangle_{2,\mu}| \lesssim \norm{\mathbb{D}u}_{\L^p_{\mu_p}(\R^{n+1})} \norm{\phi}_{\L^{p'}_{\mu_{p'}}(\R^{n+1})}.
\end{equation*}
The result follows once one proves that $\mathrm{ran}(\sqrt{\H^\star}) \cap \L^{p'}_{\mu_{p'}}(\R^{n+1})$ is dense in $\L^2_\mu(\R^{n+1}) \cap \L^{p'}_{\mu_{p'}}(\R^{n+1})$. This follows exactly as in \cite[Lemma 7.2]{AEbook2023} since $p'\in (p_-(\H^\star),2]$. 

\qed 

\begin{rem}
If $\omega \in RH_{1+\frac{2}{n}}(\R^n)$, one may directly take $\phi \in \Cont_0^\infty(\R^{n+1})$ by Corollary~\ref{cor: many functions}.
\end{rem}

\section{Final comments and open questions}\label{s8}

We conclude with some comments and open questions.
\newline
\paragraph{\textit{\textbf{1- Weak-type estimates at the endpoints $p_-(\H)$.}}}
In all the previous results, except for $\omega^{\frac{1}{2}}\nabla_x \mathcal{H}^{-1/2}$ when $\omega \in RH_{\infty}(\R^n)$ with real coefficients, it is not known whether weak-type $(p,p)$ estimates for $\mathcal{R}_{\H}$ hold at the endpoints $p_-(\H)$. For instance, the weak-type $(1,1)$ estimate for $\omega^{\frac{1}{2}} D_t^{1/2} \mathcal{H}^{-1/2}$ remains open in the case of real coefficients, even when $\omega=1$.
\newline
\paragraph{\textit{\textbf{2- Weak-type $(1,1)$ for $\omega^{\frac{1}{2}} \nabla_x \mathcal{H}^{-1/2}$ for $\omega \in RH_{\star}(\R^n)$ and real coefficients.}}}
We know that $p_-(\H)=1$, but it is not clear whether the spatial component $\omega^{\frac{1}{2}} \nabla_x \mathcal{H}^{-1/2}$ extends to a bounded operator from $\L^1_{\mu_1}(\R^{n+1})$ to $\L^{1,\infty}(\R^{n+1})$, or whether such boundedness is specific to weights in $RH_{\infty}(\R^n)$.
\newline
\paragraph{\textit{\textbf{3- Optimal upper bound for $p_-(\H)$ when $\omega\in RH_{1+\frac{2}{n}}(\R^{n})$.}}}
We know that $p_-(\H)\le 2_\star$. It is reasonable to expect that $p_-(\H)\le 2_\star-\varepsilon_{\H}$ for some $\varepsilon_{\H}>0$, as in the unweighted case \cite[Theorem 6.1]{BEK26}. Such a result would follow from an application of Shneiberg's stability theorem~\cite{Shneiberg}. However, complex interpolation for the weighted parabolic Sobolev spaces $\L^p(\R;\mathrm{W}_{\omega^{\nicefrac{p}{2}}}^{1,p}(\R^n))\cap \mathrm{H}^{\frac{1}{2},p}(\R;\L^p_{\omega^{\nicefrac{p}{2}}}(\R^n))$ is not known, and one would also need an $\L^p$-version of the weighted parabolic Sobolev embeddings in Proposition~\ref{lem: Sobolev} as a second ingredient; see \cite[Appendix A]{BEK26} and \cite{GopalaRao} for the unweighted case. Let us just mention that $\varepsilon_{\H}>0$ is best possible for all parabolic operators when $n\ge2$, as is already the case in the unweighted setting \cite[Proposition 9.1]{BEK26}.
\newline
\paragraph{\textit{\textbf{4- Relation between $p_-(\H)$ and $q_-(\H)$.}}}
Although Proposition~\ref{prop: p_=q_} suffices for our purposes, it is natural to ask whether equality holds, as in the unweighted case. This remains open and would require an $\L^p$ weighted parabolic Sobolev embedding, as discussed above.
\newline
\paragraph{\textit{\textbf{5- Reverse inequalities.}}} Characterize the range of $p$ for which \eqref{eq: reverse} holds.
\newline
\paragraph{\textit{\textbf{6- Case $p>2$.}}}
The case $p>2$ remains open, even when $\omega=1$. One would expect an upper exponent $q_+(\H)>2$, as in the elliptic case \cite{auscher2007necessary,cruz2017kato}, but this is not yet understood.
\newline
\paragraph{\textit{\textbf{7- Case where $\omega$ is the ambient measure.}}}
Recall that $\mathrm{d}\mu(x,t) = \omega(x)\,\mathrm{d}x\,\mathrm{d}t$ on $\R^{n+1}$. An interesting question is whether, and for which $p<2$, the Riesz transform $\mathcal{R}_\H$ extends to a bounded operator on $\L^p_{\mu}(\R^{n+1})$. It is reasonable to expect this to hold for some range $1<p<2$. More precisely, one would then define
\begin{align*}
    \tilde{p}_-(\H) 
    &:= \inf \left\{ p \in( 1,\infty) : \ \text{the family} \ ( \E_\lambda )_{\lambda>0} \ \text{is uniformly bounded on} \ \L_\mu^p(\R^{n+1}) \right\}, \\
    \tilde{q}_-(\H) 
    &:= \inf \left\{ p \in( 1,\infty) : \ \text{the family} \ (\lambda \mathbb{D} \E_\lambda )_{\lambda>0} \ \text{is uniformly bounded on} \ \L_\mu^p(\R^{n+1})  \right\},
\end{align*}
and expect $\mathcal{R}_\H$ to be bounded on $\L^p_{\mu}(\R^{n+1})$ for each $p\in (\tilde{p}_-(\H),2)$. The strategy would be to replace the unweighted $\L^p-\L^q$ off-diagonal estimates, conjugated by $\omega^{\frac{1}{2}}$ and $\omega^{-\frac{1}{2}}$, with weighted $\L^p_\mu-\L^q_\mu$ off-diagonal estimates on parabolic cubes and stretched annuli (by $2^j$ in $x$ and $N^j$ in $t$), denoted by $\mathcal{O}(\L^p_\mu-\L^q_\mu)$, by taking averages as in the elliptic case \cite{cruz2017kato}. We are unable to prove them at this time and we explain what can be done and the difficulties to conclude. First, let us note that the off-diagonal estimates stated in Theorem \ref{thm: parabolicODEs} still hold when $\omega$ is interpreted as the ambient measure, by exactly the same proof:
\begin{prop}\label{prop: odes w ambient}
Let $\varrho \in (\max(\tilde{p}_-(\H),\tilde{q}_-(\H)),2]$. Let $p\in(\varrho,2]$ or $p=\varrho=2$.
Fix $N\ge 4$ and an integer $m \ge 1$. Then there exists a constant $C$ such that the off-diagonal estimates
\begin{align*}
    \|\one_F \,  \lambda \D\mathcal{E}^m_\lambda \left( \one_E u \right) \|_{\L_\mu^p(\R^{n+1})}
        \le C\biggl(\frac{\lambda}{r} +\Bigl(\frac{\lambda}{r}\Bigr)^{4N}\biggr)
        N^{-j \varepsilon}
        \|\one_E u\|_{\L_\mu^p(\R^{n+1})}
\end{align*}
holds for all $\lambda, r >0$, all $j \ge 2$ and all $u \in \L_\mu^p(\R^{n+1}) \cap \L_\mu^2(\R^{n+1})$, with $E$, $F$, and $\varepsilon$ as in scenarios (i), (i) bis, and (ii) of Theorem \ref{thm: parabolicODEs}.
\end{prop}

In the case of complex coefficients, the difficulty is that there seems to be no way to prove $\tilde{p}_-(\H)<2$ using only the available information on $\L^2_\mu(\R^{n+1})$: boundedness of the parabolic gradient of the resolvent family, and boundedness and off-diagonal estimates for the resolvent family. In the elliptic case, one uses a local Sobolev inequality \cite[Theorem~15.26]{heinonen2018nonlinear} yielding higher integrability, always with respect to $\d\omega$, together with a duality argument. In the parabolic case, such an inequality is not known for the parabolic gradient $\D=(\nabla_x,D_t^{1/2})$ yielding higher integrability with respect to $\d\mu$, not $\d\mu_{q'}$ as in Proposition \ref{lem: Sobolev}.

In the case of real coefficients, Gaussian bounds (Lemma~\ref{lem: bornes gaussiennes}) are available, and no Sobolev inequality argument is needed. In this case, as expected, we can show that
\begin{equation*}
    \tilde{p}_-(\H)=\tilde{q}_-(\H)=1.
\end{equation*}
We briefly explain the proof. By the upper Gaussian bounds in (1) of Lemma~\ref{lem: bornes gaussiennes}, we have
\begin{equation*}
    |\E_\lambda u(x,t)| \le C (M_t M_\omega u)(x,t),
\end{equation*}
for some constant $C>0$, all $\lambda>0$, all $u\in \L^2_\mu(\R^{n+1})$, and almost all $(x,t)\in\R^{n+1}$, where $M_tM_\omega$ denotes the iterated maximal function on $\R^{n+1}$ with respect to $\d t$ and $\d\omega$. This implies that $\tilde{p}_-(\H)=1$. To prove that $\tilde{q}_-(\H)=1$, several non-trivial arguments are required. The first is the following boundedness result:
\begin{lem}
    There exists an integer $m\ge 1$ such that for all $p,q\in (1,\infty)$ with $p<q$:
    \begin{equation*}
        \norm{\E^m_\lambda \one_{C_j^N(\Delta_\lambda)}u}_{L^q_\mu(\R^{n+1})} \le C N^{2jm} \mu(C_j^N(\Delta_\lambda))^{\frac{1}{q}-\frac{1}{p}} \norm{\one_{C_j^N(\Delta_\lambda)}}_{L^p_\mu(\R^{n+1})},
    \end{equation*}
    where $C$ is a constant depending on $N$, $p$, $q$, and the structural constants.
\end{lem}
Combining this result with Proposition \ref{prop: odes w ambient}, a variant of Proposition \ref{prop: MtMx}, and a variant of Lemma \ref{lem: sans m}, we can implement an iteration argument starting from $2$ to prove that $\tilde{q}_-(\H)=1$.

Let us now explain where our extrapolation argument for the boundedness of $R_\H$ gets blocked when following the proof strategy in the unweighted case \cite[Theorem 7.3.]{BEK26}. First, an adapted version of the extrapolation criterion, Theorem~\ref{thm: BK}, is available. When applying it, one needs three $\L^p_\mu-\L^q_\mu$ boundedness estimates for $\E_\lambda^\beta$. Using Gaussian bounds, one then needs to estimate, with an upper bound involving $\norm{\one_E u}^q_{\L^p_\mu(\R^{n+1})}$, the quantity given by
\begin{equation*}
    \iint_{F} \left( \iint_E \mathbf{1}_{t>s} \frac{e^{-\frac{t-s}{\lambda^2}}}{\lambda^{2\beta}}(t-s)^{(\beta-1)}\frac{e^{-\frac{1}{c}\frac{|x-y|^2}{t-s}}}{\omega_{t-s}(y)} |u(y,s)| \d \mu(y,s)  \right)^{q} \d \mu(x,t).
\end{equation*}
Set $\Delta^{j-\varepsilon,\delta}_r:=2^{j-\varepsilon}Q_r \times N^{\delta j}I_r$. For the first estimate, $E=\R^{n+1}$ and $F=\Delta_r^{j-\varepsilon,\delta}$; for the second, $E=\R^{n+1}\setminus\Delta_r^{j-2\varepsilon,\delta/2}$ and $F=\R^{n+1}\setminus\Delta_r^{j-\varepsilon,\delta}$; and for the third, $E=\R^{n+1}$ and $F=\R^{n+1}\setminus\Delta_r^{j-\varepsilon,\delta}$. The first and second estimates can be obtained using elementary H\"older inequalities, the finite measure of $F$ for the first, and the positive parabolic distance between $E$ and $F$ for the second. However, we do not know how to handle the third estimate, as $F$ has infinite measure and $F\subset E$. In the unweighted case, one uses Young's convolution inequality, which is also the spirit of the argument when the measure is interpreted as a multiplication operator.

\appendix

\section{Proof of Lemma \ref{lem: bornes gaussiennes}}\label{annexe 1}

Fix $\lambda \in \C$ such that $\mathrm{Re}(\lambda^{-2})>0$, and let $f \in \L^2_\mu(\R^{n+1})$. We treat the case $m=1$ and explain in the final step how to obtain the result for general $m$. We set $u:= \E_\lambda f = (1+\lambda^2 \H)^{-1} f$. Then $u \in \mathrm{E}_\mu$ and $u$ is a weak solution to the equation
\begin{equation}\label{eq: parabolic eq}
    \partial_t u -\omega^{-1} \mathrm{div}_x(A\nabla_x u)+\lambda^{-2}u = \lambda^{-2}f. 
\end{equation}
Let $\Gamma = \Gamma(t,s)_{t,s \in \R}$ be the fundamental solution for $\H$. We refer to \cite{baadi2025degenerate} for its definition, as well as for the notion of weak solutions and the distributional duality in this degenerate setting. By \cite{baadi2025degenerate,ataei2024fundamental}, we know that for all $t > s$, $\Gamma(t,s)$ is an integral operator on $\L^2_\omega(\R^n)$ with kernel $\Gamma(x,t;y,s)$, having bounds
\begin{equation}\label{eq: bounds}
   \frac{1}{C} \,\mathbf{1}_{t>s}
\frac{e^{-c\frac{|x-y|^2}{t-s}}}{\omega_{t-s}(y)}  \le \Gamma(x,t;y,s)
\le C \,\mathbf{1}_{t>s}
\frac{e^{-\frac{1}{c}\frac{|x-y|^2}{t-s}}}{\omega_{t-s}(y)},
\end{equation}
where $C>1$ and $c>1$ are constants depending only on $M$, $\nu$, $[\omega]_{A_2}$, and $n$. Note that $\Gamma(t,t) = I_d$, the identity operator on $\L^2_\omega(\R^n)$ and $\Gamma(t,s) = 0$ whenever $t < s$. For all $(x,t)\in \R^{n+1}$, we set 
\begin{equation*}
    v(x,t):=\iint_{\R^{n+1}} \lambda^{-2} \Gamma(x,t;y,s) e^{-\frac{t-s}{\lambda^2}}f(y,s) \d y \d s = \int_{-\infty}^t \lambda^{-2} e^{-\frac{t-s}{\lambda^2}}  \left ( \Gamma(t,s)f(\cdot,s) \right )(x) \d s.
\end{equation*}
$v$ is in fact the representation formula with the fundamental solution operator for $\H + \lambda^{-2}$ for the solution $u$ to the equation \eqref{eq: parabolic eq}. We provide the relevant details for completeness. 
\newline
\paragraph{\textit{\textbf{Step 1:} $v\in \L^2_\mu(\R^{n+1})$, $\nabla_x v \in \L^2_\mu(\R^{n+1})^n$ and $v$ solves \eqref{eq: parabolic eq}}} Using \eqref{eq: bounds} with $\omega_{t-s}(x)$ instead of $\omega_{t-s}(y)$ (see Remark \ref{rem: Cruz}), we easily show that, for all $(x,t) \in \R^{n+1}$, one has
\begin{equation*}
    |v(x,t)| \le C (M_t M_\omega f)(x,t)
\end{equation*}
where $C = C(M,\nu,[\omega]_{A_2},n)$ is a constant, and $M_t M_\omega$ denotes the iterated maximal function on $\R^{n+1}$ with respect to $\d t$ and $\d \omega$. Thus, $v\in \L^2_\mu(\R^{n+1})$. For the other one, by the properties of the fundamental solution, Fubini's theorem, and the Cauchy–Schwarz inequality, we have
\begin{align*}
    \norm{\nabla_x v}^2_{2,\mu} &= \iint_{\R^{n+1}}  \left | \int_{-\infty}^t \lambda^{-2} e^{-\frac{t-s}{\lambda^2}}  \left ( \nabla_x \Gamma(t,s)f(\cdot,s) \right )(x) \d s \right |^2 \d \omega(x) \d t 
    \\& \le |\lambda|^{-4} \iint_{\R^{n+1}} \left ( \int_{-\infty}^t  e^{-2\mathrm{Re}(\lambda^{-2}) (t-s)} \d s \right ) \left ( \int_{-\infty}^t |\nabla_x (\Gamma(t,s)f(\cdot,s))(x)|^2 \d s \right ) \d \omega(x) \d t 
    \\& = \frac{|\lambda|^{-4}}{2\mathrm{Re}(\lambda^{-2})}  \int_\R \left ( \int_{s}^{+\infty} \left \| \nabla_x \Gamma(t,s)f(\cdot,s) \right \|^2_{2,\omega} \d t \right ) \d s
    \\& \le C \frac{|\lambda|^{-4}}{2\mathrm{Re}(\lambda^{-2})} \int_\R \left \| f(\cdot,s) \right \|^2_{2,\omega} \d s =  \frac{C|\lambda|^{-4}}{2\mathrm{Re}(\lambda^{-2})} \norm{f}_{2,\mu}^2,
\end{align*}
where we used \cite[Theorem 2, 1]{baadi2025degenerate} in the last inequality, and $C = C(M,\nu)$ is a constant. Finally, we easily verify, by a direct computation using the properties of the fundamental solution, that $v$ solves \eqref{eq: parabolic eq}. 
\newline
\paragraph{\textit{\textbf{Step 2:} concluding the case $m=1$. }} Set $h = v - u$. Then $h \in \L^2_\mu(\R^{n+1})$, $\nabla_x h \in \L^2_\mu(\R^{n+1})^n$, and $h$ satisfies in the weak sense
\begin{equation*}
    \partial_t h - \omega^{-1} \mathrm{div}_x(A\nabla_x h) + \lambda^{-2} h = 0.
\end{equation*}
We then easily conclude that $h = 0$, either by the invertibility results for the above parabolic operator \cite{baadi2026wellposedness}, or by the Lions embedding theorem \cite{auscherbaadi2024fundamental}, by writing the energy identity and using the fact that $h \in C(\R; \L^2_\omega(\R^n))$ with limit zero as $t \to \pm \infty$, together with the ellipticity of $A$ and the assumption $\mathrm{Re}(\lambda^{-2})>0$. Thus, we have
\begin{equation*}
    K^1_\lambda(x,t;y,s) = \frac{e^{-\frac{t-s}{\lambda^2}}}{\lambda^2}\Gamma(x,t;y,s).
\end{equation*}
For $\lambda>0$, we directly conclude using the two-sided bounds in \eqref{eq: bounds}. For $\lambda \in S^+_{\mu}$, we use the upper bound in \eqref{eq: bounds} together with the fact that
\begin{equation*}
   \mathrm{Re}(\lambda^{-2})= \nicefrac{\mathrm{Re}(\lambda^2)}{|\lambda|^4} \ge \frac{\cos(2\mu)}{|\lambda|^2},
\end{equation*}
and the result follows with $c_\mu=\cos(2\mu)$.
\newline
\paragraph{\textit{\textbf{Step 3:} the case of general $m$}} For simplicity, we only treat the case $m=2$ for the upper bound when $\lambda>0$, since the lower bound and the general case follow similarly. By composition and Fubini's theorem, $\E_\lambda^2$ is an integral operator on $\L^2_\mu(\R^{n+1})$ with kernel $K^2_\lambda(x,t;y,s)$ given by the first equality below:
\begin{align*}
    0 \le K^2_\lambda(x,t;y,s)&=\iint_{\R^{n+1}} K^1_\lambda(x,t;z,\ell) K^1_\lambda(z,\ell;y,s) \d \omega(z) \d \ell 
    \\& \le C^2 \frac{e^{-\frac{t-s}{\lambda^2}}}{\lambda^4} \int_{s}^{t}\left( \int_{\R^n} \frac{e^{-c\frac{|x-z|^2}{t-\ell}}}{\sqrt{\omega_{t-\ell}(x)}\sqrt{\omega_{t-\ell}(z)}} \frac{e^{-c\frac{|z-y|^2}{\ell-s}}}{\sqrt{\omega_{\ell-s}(z)}\sqrt{\omega_{\ell-s}(y)}} \d \omega(z) \right) \d \ell.
\end{align*}
We fix $\ell \in (s,t)$. We know that, for all $r>0$, the degenerate heat semigroup $e^{r\Delta_\omega}$ has a kernel $\mathcal{K}_{r}(x,y)$ with two-sided Gaussian bounds \cite{baadi2025degenerate}: there exist two constants $K_0 = K_0([\omega]_{A_2}, n) > 1$ and $k_0 = k_0([\omega]_{A_2}, n) > 1$ such that, for almost all $x,y \in \R^n$,
\begin{equation*}
   K_0^{-1}\frac{e^{-\frac{k_0 |x-y|^2}{r}}}{\sqrt{\omega_{r}(x)}\sqrt{\omega_{r}(y)}} \le \mathcal{K}_{r}(x,y) \le K_0 \frac{e^{-\frac{|x-y|^2}{k_0 r}}}{\sqrt{\omega_{r}(x)}\sqrt{\omega_{r}(y)}} .
\end{equation*}
Thus, using first the above lower bound, then the additivity of the semigroup, and finally the above upper bound, we have
\begin{align*}
    0 \le K^2_\lambda(x,t;y,s)&\lesssim K_0^2 C^2 \frac{e^{-\frac{t-s}{\lambda^2}}}{\lambda^4} \int_{s}^{t}\left( \int_{\R^n}  \mathcal{K}_{\frac{k_0}{c}(t-\ell)}(x,z) \mathcal{K}_{\frac{k_0}{c}(\ell-s)}(z,y) \d \omega(z) \right) \d \ell
    \\&= K_0^2 C^2 \frac{e^{-\frac{t-s}{\lambda^2}}}{\lambda^4} \left( \int_{s}^{t} 1 \d \ell \right) \mathcal{K}_{\frac{k_0}{c}(t-s)}(x,y) 
    \\&\lesssim K_0^3 C^2 \frac{e^{-\frac{t-s}{\lambda^2}}}{\lambda^4} (t-s) \frac{e^{-\frac{c}{k_0^2}\frac{|x-y|^2}{t-s}}}{\sqrt{\omega_{t-s}(x)}\sqrt{\omega_{t-s}(y)}}.
\end{align*}
Here, we have used the notation $\lesssim$ since we have deliberately ignored, for $r>0$, the comparison between $\omega_r(x)$ and $\omega_{\frac{k_0}{c} r}(x)$ for simplicity. In fact, these quantities are comparable: there exist constants $\eta \in (0,1)$ and $\beta > 0$, depending only on $n$ and $[\omega]_{A_2}$, such that
\begin{equation*}
\beta^{-1}\left( \frac{|E|}{|B|} \right)^{\frac{1}{2\eta}} \leq \frac{\omega(E)}{\omega(B)} \leq \beta\left( \frac{|E|}{|B|} \right)^{2\eta},
\end{equation*}
whenever $B \subset \R^n$ is an Euclidean ball and for all measurable sets $E \subset B$. See \cite[Chapter 15.5]{heinonen2018nonlinear}. \qed

\subsubsection*{\textbf{Copyright}}
A CC-BY 4.0 \url{https://creativecommons.org/licenses/by/4.0/} public copyright license has been applied by the authors to the present document and will be applied to all subsequent versions up to the Author Accepted Manuscript arising from this submission.

\bibliographystyle{alpha}
\bibliography{references.bib}

\end{document}